\documentclass[12pt,epsfig,amsfonts]{amsart}
\usepackage{amsmath,amsthm,amssymb,amscd,epsfig,color}
\usepackage{graphicx}

\newcommand{\Hc}{H}
\usepackage{url}
\usepackage[all]{xy}
\usepackage{mathrsfs}
\usepackage{ulem}
\usepackage{pb-diagram} 

\newtheorem*{mainthm}{Theorem A}
\newtheorem*{dense-cor}{Theorem B}
\newtheorem*{code-thm}{Theorem C}

\newtheorem{prop}{Proposition}[section]
\newtheorem{lemma}[prop]{Lemma}

\newtheorem{thm}[prop]{Theorem}
\newtheorem{cor}[prop]{Corollary}

\theoremstyle{definition}

\theoremstyle{remark}
\newtheorem{remark}[prop]{Remark}

\numberwithin{equation}{section}

\begin{document}

\author{Mao Shinoda, Hiroki Takahasi, Kenichiro Yamamoto}

\address{Department of Mathematics,
Ochanomizu University, 2-1-1 Otsuka, Bunkyo-ku, Tokyo, 112-8610, JAPAN} 
\email{shinoda.mao@ocha.ac.jp}

\address{Department of Mathematics,
Keio University, Yokohama,
223-8522, JAPAN} 
\email{hiroki@math.keio.ac.jp}

\address{Department of General Education, Nagaoka University of Technology, Nagaoka 940-2188, JAPAN}
\email{k\_yamamoto@vos.nagaokaut.ac.jp}

\thanks{}

\date{}

 \title[Ergodic optimization for continuous functions]
 {
 Ergodic optimization 
 for continuous functions on non-Markov shifts}

 \maketitle

\begin{abstract}
Ergodic optimization aims to describe dynamically invariant probability measures that maximize the integral of a given function. For a wide class of intrinsically ergodic subshifts  over a finite alphabet, we show that
the space of continuous functions on the shift space splits into two subsets: one is a dense $G_\delta$ set for which all maximizing measures have `relatively small' entropy; the other 
is contained in the closure of the set of functions having uncountably many, fully supported ergodic measures with `relatively large' entropy. 
This result considerably generalizes and unifies the results of  
Morris (2010) and
Shinoda (2018),  
and applies to a wide class of intrinsically ergodic
non-Markov symbolic dynamics 
without Bowen's specification property, including 
any transitive piecewise monotonic interval map, some coded shifts
and multidimensional $\beta$-transformations.
Along with these examples of application, we provide an example of an 
 intrinsically ergodic subshift with positive obstruction entropy to specification.
\end{abstract}


 \section{Introduction}
 Ergodic optimization aims to describe 
 properties of dynamically invariant
maximizing measures. In its most basic form, main constituent components are:
a continuous map $T$ of a compact metric space $X$; 
the space $M(X,T)$ of $T$-invariant Borel probability measures endowed with the weak* topology;
a continuous function $f\colon X\to\mathbb R$. Elements of $M(X,T)$ that attain the supremum \begin{equation}\label{alphaT}\Lambda_T(f)=\sup\left\{\int f{\rm d}\mu\colon\mu\in M(X,T)\right\}\end{equation}
are called 
{\it $f$-maximizing measures}.
The set of $f$-maximizing measures, denoted by $M_{\rm max}(f)$, is non-empty and contains ergodic measures.
For a given dynamical system $(X,T)$ and a Banach space of real-valued functions on $X$, we aim to establish properties of elements of $M_{\rm max}(f)$
for a `typical' function $f$ in the space.

Earlier results in ergodic optimization, especially those by Hunt and Ott \cite{HT1,HT2} and Bousch \cite{Bou00}, suggested the possibility of maximizing measures for typical functions being unique and supported on periodic orbits. It has been conjectured that for typical dynamical systems 
and typical 
functions in suitable Banach spaces, the maximizing measure 
is unique and supported on a periodic orbit. For a more precise statement known as the TPO (typically periodic optimization) conjecture, see \cite[Conjecture~4.11]{Jen06},  \cite[Conjecture~7.3]{Jen19} and  \cite[Conjecture~1]{YH99}. For results toward a resolution of this conjecture in a variety of settings, see \cite{Boc18,Jen06,Jen19} and the references therein.

In ergodic optimization, 
the regularity of functions is crucial. For $(X,T)$ with some 
expanding or hyperbolic behavior and a H\"older continuous $f$,
the Ma\~n\'e-Conze-Guivarc'h lemma 
characterizes $f$-maximizing measures via their supports \cite{Bou00,Bou01,CLT01,Mor09}.
The analysis of functions in the space 
$C(X)$ of real-valued continuous functions on $X$ endowed with the supremum norm $\|\cdot\|_{C^0}$ is completely different: the Ma\~n\'e-Conze-Guivarc'h lemma is not valid, but duality arguments are available.

Recall that dense $G_\delta$ sets are countable intersections of open dense subsets, and a property that holds for a dense $G_\delta$ set is said to be generic. 
For an expanding map of the circle,
Bousch and Jenkinson \cite{BouJen02} proved that the maximizing measure is unique and fully supported (charging any non-empty open set) for generic continuous functions. Br\'emont \cite{Bre08} proved that if $(X,T)$ has the property that 
measures supported on periodic orbits (closed orbit measures) are dense in $M(X,T)$ and the entropy function is upper semicontinuous, then the maximizing measure is unique and has zero entropy for generic continuous functions.
For expanding maps of the circle,  Bowen's specification property \cite{Bow71} holds
and so closed orbit measures are dense \cite{Sig70}. Moreover the entropy function is upper semicontinuous. Hence, the result of
Br\'emont \cite{Bre08}  applies 
together with that of Bousch and Jenkinson \cite{BouJen02}: the maximizing measure is unique, fully supported. has zero entropy, and is not strongly mixing for generic continuous functions.

Morris \cite{Mor10} unified the result of Bousch and Jenkinson \cite{BouJen02} and that of Br\'emont \cite{Bre08}, proving that for $(X,T)$ with Bowen's specification property, the maximizing measure is unique, fully supported, has zero entropy and is not strongly mixing for  generic continuous functions \cite[Corollary~2]{Mor10}. 
In contrast to Morris's result, 
Shinoda \cite[Theorem~A]{Shi18} proved that 
for a dense set of continuous functions on
a topologically mixing Markov shift,
 there exist uncountably many, fully supported ergodic maximizing measures with positive entropy, which are actually Bernoulli. For an analogous result on expanding Markov interval maps with holes, see \cite{ShiTak20}.



\subsection{Statements of the results}
Our main result considerably generalizes, and unifies 
Morris's result \cite[Corollary~2]{Mor10} and 
Shinoda's pathological one \cite[Theorem~A]{Shi18} concerning entropies and supports of maximizing measures for continuous functions, in the context of symbolic dynamics to include a wide class of non-Markov shifts
without Bowen's specification property.
  For a subshift $\Sigma$, let $h_{\rm top}(\Sigma)$ denote its topological entropy, and 
    for a shift-invariant Borel probability measure $\mu$ on $\Sigma$ let $h(\mu)$ denote its measure-theoretic entropy.
    Let $h_{\rm spec}^\bot(\Sigma)$
    denote {\it the obstruction entropy to specification} on $\Sigma$.
  Our main result is stated as follows.

\begin{mainthm}
Let $\Sigma$ be a subshift that satisfies $h_{\rm spec}^\bot(\Sigma)<h_{\rm top}(\Sigma)$.
If $h_{\rm spec}^\bot(\Sigma)\leq H<h_{\rm top}(\Sigma)$ then the following statements hold:
\begin{itemize}
\item[(a)] The set
$\mathscr{R}_{H }=\{f\in C(\Sigma)\colon h(\mu)\leq H
\text{ for all $\mu\in M_{\rm max}(f)$}\}$
is dense $G_\delta$. 
\item[(b)] For any $f\in C(\Sigma)\setminus \mathscr{R}_H$
 and any neighborhood $U$ of $f$ in $C(\Sigma)$, there exists $g\in U$ 
 such that the set 
 $\{\mu\in M_{\rm max}(g)\colon h(\mu)>H\}$
contains uncountably many, fully supported ergodic measures.
\end{itemize}
\end{mainthm}



The condition $h_{\rm spec}^\bot(\Sigma)<h_{\rm top}(\Sigma)$, together with the obstruction entropy to specification, was introduced by Climenhaga and Thompson in the series of papers \cite{CT12,CT13,CT14} on the thermodynamic formalism for dynamical systems with non-uniform specification. It implies that for any $H\in[h_{\rm spec}^\bot(\Sigma),h_{\rm top}(\Sigma))$ the language $\mathcal L(\Sigma)$ of the shift space admits a decomposition $\mathcal L(\Sigma)=\mathcal C^p\mathcal G\mathcal C^s$ into a prefix $\mathcal C^p$, a `good' core $\mathcal G$ and a suffix $\mathcal C^s$, such that on a `fattened' $\mathcal G$ there is a specification, and the exponential growth rate of the cardinality of $\mathcal C^p\cup\mathcal C^s$ per word length does not exceed $H$.
See $\S\ref{TF}$ for the formal definition.
This 
implies the intrinsic ergodicity of $\Sigma$ \cite[Theorem~C]{CT13}, i.e., the uniqueness of the measure of maximal entropy.

As a consequence of the variational principle and the result of Jenkinson \cite[Theorem~3.7]{Jen06}, the set $C(\Sigma)\setminus\mathscr{R}_{H }$ is non-empty.
It is a challenging problem to give a concrete example 
of a continuous function 
for which there exist uncountably many maximizing measures that are fully supported and have positive entropy. 
Trivial examples of 
with 
this property are 
functions
that are homologous to a constant. The Ma\~n\'e-Conze-Guivarc'h lemma implies that any continuous function
that is not homologous to a constant and
has a fully supported maximizing measure cannot be H\"older continuous.


Continuous functions like $g$ in (b) are somewhat pathological. The theorem below provides a sufficient condition for the density of such functions.
Let $\Sigma$ be a subshift and let $\sigma$ denote the left shift acting on $\Sigma$.
We say {\it ergodic measures on $\Sigma$ are entropy dense} if 
for any non-ergodic measure $\mu\in M(\Sigma,\sigma)$, any $\varepsilon>0$ and any neighborhood $U$ of $\mu$ in $M(\Sigma,\sigma)$, there exists an ergodic measure $\nu\in U$ such that $h(\nu)>h(\mu)-\varepsilon$. 
Clearly, the entropy density implies the density of ergodic measures.


\begin{dense-cor}\label{dense-cor}
Let $\Sigma$ be a subshift with positive topological entropy that satisfies $h_{\rm spec}^\bot(\Sigma)=0$.
Suppose that ergodic measures on $\Sigma$ are entropy dense.
Then, there exists a dense subset $\mathscr{D}$ of $C(\Sigma)$
such that for any $f\in \mathscr{D}$,
 $M_{\rm max}(f)$
contains uncountably many fully supported ergodic measures with positive entropy.
\end{dense-cor}

There is a wealth of examples of symbolic dynamics to which Theorem~A or Theorem~B applies:
 piecewise monotonic maps \cite{H2}, some coded shifts including
$S$-gap shifts \cite{LM}, multidimensional $\beta$-transformations \cite{Buz97,Buz05}, and so on.
See $\S4$ for more details on these applications.

We will provide two examples of intrinsically ergodic subshifts to which Theorem~A applies but Theorem~B does not.
One is a subshift without the entropy density (see $\S\ref{coded-sec}$). The other example is a subshift with positive obstruction entropy to specification, which is stated as follows.

\begin{code-thm}\label{code-thm}
For any integer $N\geq3$,
there exists a subshift $\Sigma$ on $N$ symbols such that $0<h_{\rm spec}^\bot(\Sigma)<h_{\rm top}(\Sigma)$. 
\end{code-thm}

Theorem~C has an independent interest.
For any transitive Markov shift (hence intrinsically ergodic), the obstruction entropy to specification is $0$.
The converse is not true:
There is a number of examples of intrinsically ergodic non-Markov shifts for which the obstruction entropy to specification is $0$. 
Intrinsically ergodic subshifts with positive obstruction entropy to specification have received little attention and should be investigated further.
Under the hypothesis of Theorem~A,
measures whose entropy do not exceed the obstruction entropy to specification play no role in search of measures of maximal entropy, but they do play roles in problems concerning rare events not captured by measures of maximal entropy,
such as large deviations and multifractal analysis. 

\subsection{Structure of the paper} 
For proofs of (a) and (b) in Theorem~A, we develop 
ideas of Morris \cite[Theorem~1.1, Corollary~2]{Mor10} and Shinoda \cite[Theorem~A]{Shi18} respectively, both related to approximations of ergodic measures.
On the one hand, a key observation is that
 the assumption of Bowen's specification property in \cite[Corollary~2]{Mor10} can actually be weakened to 
 a property that any ergodic measure can be approximated by another with arbitrarily small entropy.
 On the other hand, a key ingredient in the proof of \cite[Theorem~A]{Shi18} is the construction of paths of ergodic measures of high complexity in arbitrarily small neighborhoods of maximizing measures. This construction relies on Sigmund's result \cite{Sig77} that is not valid for non-Markov shifts. 

 The rest of this paper consists of four sections.
The most critical component is Proposition~\ref{hyp-lem2} 
that provides two different kinds of approximations of ergodic measures in the weak* topology.
The proof of Proposition~\ref{hyp-lem2}
relies on the thermodynamic formalism for subshifts with non-uniform specification \cite{CT13} and the zero temperature limit \cite{BLL} that we recall in $\S2$.
     In $\S3$ we combine results in $\S2$ with part of the proof of \cite[Theorem~1.1]{Mor10} and a duality argument in \cite{Isr79} to complete the proofs of Theorem~A and Theorem~B. In $\S4$ we provide examples of application of these two theorems:
any transitive piecewise monotonic interval map; some coded shifts including $S$-gap shifts;
some multidimensional $\beta$-transformations.
          In $\S5$ we prove Theorem~C by constructing concrete examples.

\section{Preliminaries}
This section contains 
preliminary results needed for the proof of Theorem~A. In $\S\ref{nice}$ we begin by introducing basic terminologies in symbolic dynamics. In $\S$\ref{TF} we recall the result of Climenhaga and Thompson \cite{CT13} on the thermodynamic formalism for subshifts with non-uniform specification. In $\S$\ref{obst-sec} we introduce the obstruction entropy to specification, and
in $\S$\ref{zero-t} recall the notion of ground states in zero temperature limit.
In $\S$\ref{construct-sec} we state and prove Proposition~\ref{hyp-lem2} that is a key ingredient in the proofs of Theorems~A and B.

\subsection{Subshift, word, language}\label{nice}
Let $A$ be a non-empty finite discrete set.
A finite string $w=w_1w_2\cdots w_n$ of elements of $A$ is called a {\it word} of length $n$ in $A$. 
For convenience, we introduce an empty word $\emptyset$ by the rules
  $\emptyset w=w\emptyset=w$ for any word $w$ in $A$ and $\emptyset\emptyset=\emptyset$.
  The word length of $w\in A$ is denoted by $|w|$, and the word length of the empty word is set to be $0$. 
  
Given a collection $\mathcal L$ of  words in $A$ and the empty word,
we consider its decompositions: collections of words $\mathcal C^p$, $\mathcal G$, $\mathcal C^s\subset\mathcal L$ such that
$\mathcal C^p\mathcal G\mathcal C^s=\mathcal L$. This means that every word in $\mathcal L$
can be written as a concatenation of a prefix (from $\mathcal C^p$), a `good' core (from $\mathcal G$) and a suffix (from $\mathcal C^s$). Given such a decomposition, we consider for every $M\in\mathbb N$ the following `fattened' set of good words:
\begin{equation}\label{GM-def}\mathcal G^M=\{uvw\in\mathcal L\colon u\in\mathcal C^p, v\in\mathcal G, w\in\mathcal C^s,\ |u|,|w|\leq M\}.\end{equation}
Note that $\mathcal L=\bigcup_{M\in\mathbb N}\mathcal G^M$. This gives a filtration of the language.

Let $A^{\mathbb N}$ and $A^{\mathbb Z}$ denote 
the one-sided, and the two-sided Cartesian product topological spaces of $A$ respectively.
The topologies of these two spaces are metrizable with the Hamming metric $d$: for distinct points $x=(x_i)_i$, $y=(y_i)_i$,
$d(x,y)=2^{-\min\{|i|\colon x_i\neq y_i\}}.$
  The left shift acts continuously on these spaces.
  A shift-invariant closed subset of $A^{\mathbb N}$
  or $A^{\mathbb Z}$
  is called a {\it subshift} over $A$, or a subshift on $\#A$ symbols.
  A {\it language} of a subshift $\Sigma$,
denoted by $\mathcal L(\Sigma)$, is the collection of the empty word $\emptyset$ and words in $A$ that appear in some elements of $\Sigma$. 
 For a subshift $\Sigma$ and $n\geq0$, let $\mathcal L_n(\Sigma)$ denote the collection of elements of $\mathcal L(\Sigma)$ with word length $n$. Note that $\mathcal L_0(\Sigma)=\{\emptyset\}$.
 For a collection $\mathcal D\subset\mathcal L(\Sigma)$ and $n\geq0$,
put
\[\mathcal D_{n}=\mathcal D\cap\mathcal L_n(\Sigma).\]
 Note that $\mathcal D_0=\{\emptyset\}$.
Let $\sigma$ denote the left shift on $\Sigma$.

\subsection{Thermodynamic formalism}\label{TF}
Let $\Sigma$ be a subshift.
For $n\in\mathbb N$ and $w=w_1\cdots w_n\in \mathcal L_n(\Sigma)$, define the $n$-cylinder
 \[[w]=\{(x_i)_{i}\in\Sigma\colon x_i=w_{i}\text{ for }1\leq i\leq n\}.\]
Let $\phi\in C(\Sigma)$ and let 
 $\mathcal D\subset \mathcal L(\Sigma)$. For $n\in\mathbb N$, define
\[\Lambda_n(\mathcal D,\phi)=\sum_{w\in \mathcal D_n }\sup_{[w]}\exp S_n\phi,\]
where $S_n\phi=\sum_{i=0}^{n-1}\phi\circ\sigma^i,$ and put
\[h(\mathcal D,\phi)=\limsup_{n\to\infty}\frac{1}{n}\log\Lambda_n(\mathcal D,\phi).\]
In the case $\phi\equiv0$, let us abbreviate 
$h(\mathcal D)=h(\mathcal D,\phi)$. 
Note that $h(\mathcal L(\Sigma))$ equals $h_{\rm top}(\Sigma)$. The quantity $h(\mathcal L(\Sigma),\phi)$ is called the topological pressure, and denoted by $P(\phi)$.
The variational principle asserts that \[P(\phi)=\sup\left\{h(\mu)+\int\phi{\rm d}\mu\colon\mu\in M(\Sigma,\sigma)\right\}.\]
Elements of $M(\Sigma,\sigma)$ that attain this supremum are called {\it equilibrium states} for the potential $\phi$. Since $(\Sigma,\sigma)$ is expansive, the entropy function $\mu\in M(\Sigma,\sigma)\mapsto h(\mu)$ is upper semicontinuous. So, (the minus of) the free energy
$\mu\in M(\Sigma,\sigma)\mapsto h(\mu)+\int\phi{\rm d}\mu$ is upper semicontinuous and 
equilibrium states always exist.
The uniqueness is an issue.

Bowen's specification property allows us to glue pieces of orbits together to form one orbit.
This property can be used
to show the uniqueness of equilibrium states.
Since Bowen's specification property is too stringent
for non-Markov shifts, 
a number of weaker condition have been introduced and their consequences have been investigated
(see \cite{KLO16} and the references therein). 
Climenhaga and Thompson \cite{CT13} introduced the following condition, and under it proved that the measure of maximal entropy is unique, and satisfies a certain Gibbs property. 
We say a collection of words $\mathcal G\subset\mathcal L(\Sigma)$ has 
  {\it (W)-specification} if there exists an integer $t\geq0$, called a {\it gap size}, such that
 for all $m\geq2$ and all $v^{(1)},\ldots,v^{(m)}\in\mathcal G$, there exist $w^{(1)},\ldots,w^{(m-1)}\in\mathcal L(\Sigma)$ such that $v^{(1)}w^{(1)}v^{(2)}w^{(2)}\cdots w^{(m-1)}v^{(m)}\in\mathcal L(\Sigma)$ and $|w^{(i)}|\leq t$ for all $i\in\{1,\ldots,m-1\}$.

\begin{thm}
[\cite{CT13}, Theorem~C]\label{CT-thm}
Let $\Sigma$ be a subshift and let $\phi\in C(\Sigma)$
be H\"older continuous. Suppose there exist collections of words $\mathcal C^p$, $\mathcal G$, $\mathcal C^s\subset\mathcal L(\Sigma)$ such that $\mathcal L(\Sigma)=\mathcal C^p\mathcal G\mathcal C^s$, and the following two conditions hold:

\begin{itemize}
\item[(a)] $\mathcal G^M$ has (W)-specification for all $M\in\mathbb N$.

\item[(b)] The collections $\mathcal C^p$, $\mathcal C^s$ satisfy \[\sum_{n=1}^\infty\Lambda_n(\mathcal C^p\cup\mathcal C^s,\phi)e^{-P(\phi)n}<\infty.\]
\end{itemize}
Then the equilibrium state for the potential $\phi$ is unique, and has the weak Gibbs property: there exists $K>0$, and for each $M\in\mathbb N$ there exists $K_M>0$ such that for all $n\in\mathbb N$ and all $w\in\mathcal G^M\cap\mathcal L_n(\Sigma)$ we have
\begin{equation}\label{w-Gibbs}K_M\leq\frac{\mu_\phi([w])}{\exp\left( -P(\phi)n+\sup_{[w]}S_n\phi \right)}\leq K.\end{equation}
\end{thm}

\begin{remark}
More general potential functions than those in Theorem~\ref{CT-thm} were treated in \cite[Theorem~C]{CT13}, but we will not need them in this paper.
\end{remark}

\subsection{Obstruction entropy to specification}\label{obst-sec}
For a subshift $\Sigma$, we define  
\begin{equation}\label{spec-def}h_{\rm spec}^\bot(\Sigma)=\inf
\left\{
\begin{tabular}{l}
\!\!$h(\mathcal C^p\cup\mathcal C^s)\colon\mathcal L(\Sigma)=\mathcal C^p\mathcal G\mathcal C^s$ and\\
\!\!$\mathcal G^M$ has (W)-specification for all $M\in\mathbb N$\!\!\end{tabular}
\right\},
\end{equation}
and call this number
{\it the obstruction entropy to specification} on $\Sigma.$ 
Note that $h_{\rm spec}^\bot(\Sigma)\leq h_{\rm top}(\Sigma)$. By Theorem~\ref{CT-thm} with $\phi\equiv0$, the strict inequality
$h_{\rm spec}^\bot(\Sigma)< h_{\rm top}(\Sigma)$
implies the intrinsic ergodicity of $\Sigma$.

As is evident from the definition \eqref{spec-def},
the lower estimate of $h_{\rm spec}^\bot(\Sigma)$ is much harder than the upper one: to show 
$h_{\rm spec}^\bot(\Sigma)=0$ we only have to exhibit for any $\varepsilon>0$ one decomposition for which $h(\mathcal C^p\cup\mathcal C^s)<\varepsilon$, whereas to show $h_{\rm spec}^\bot(\Sigma)\geq c$ for some $c>0$ we must verify $h(\mathcal C^p\cup\mathcal C^s)\geq c$ for {\it any} decomposition $\mathcal L(\Sigma)=\mathcal C^p\mathcal G\mathcal C^s$ for which $\mathcal G^M$ has (W)-specification for all $M\in\mathbb N$.

\subsection{Entropy in zero temperature limit}\label{zero-t}
Let $\Sigma$ be a subshift and
let $f\in C(\Sigma)$. We say $\mu\in M(\Sigma,\sigma)$ is a {\it ground state} for the potential $f$ if 
there exist a sequence $(\beta_n)_{n\in\mathbb N}$ in $\mathbb R$
and a sequence $(\mu_n)_{n\in\mathbb N}$ in $M(\Sigma,\sigma)$ such that the following hold:

\begin{itemize}\item$\beta_n\to\infty$ as $n\to\infty$;

\item  $\mu_n$ is an equilibrium state for the potential $\beta_nf$ for all $n\in\mathbb N$;
 
\item $\mu_n\to\mu$ in the weak* topology as $n\to\infty$.
 \end{itemize}
 The term `ground state' has been borrowed from statistical mechanics as an analogy: since the parameter $\beta$ may be viewed as the inverse of temperature, the limit  $\beta\to\infty$ is called the zero temperature limit \cite{BLL}.

For the reader's convenience we include
 a proof of the next lemma, which can be found in \cite{BLL} for example.

\begin{lemma}\label{zero-t-lem}
Let $\Sigma$ be a subshift and 
let $\mu\in M(\Sigma,\sigma)$ be a ground state for the potential $f\in C(\Sigma)$. Let $(\beta_n)_{n\in\mathbb N}$ be a sequence in $\mathbb R$ and let $(\mu_n)_{n\in\mathbb N}$ be a sequence in $M(\Sigma,\sigma)$ such that $\beta_n\to\infty$ as $n\to\infty$, $\mu_n$ is an equilibrium state for the potential $\beta_nf$ for all $n\in\mathbb N$, and $\mu_n\to\mu$ in the weak* topology as $n\to\infty$. Then $\mu$ is $f$-maximizing, and we have
\[h(\mu)={\rm sup}\{h(\nu)\colon\nu\in M_{\rm max}(f)\}={\rm max}\{h(\nu)\colon\nu\in M_{\rm max}(f)\},\]
and
\[\lim_{n\to\infty} h(\mu_n)=h(\mu).\] 
\end{lemma}
\begin{proof}
Since $M_{\rm max}(f)$ is nonempty and closed, 
the supremum in the statement of the lemma is attained.
For all $n\in\mathbb N$ we have $h(\mu_n)+\beta_n\int f{\rm d}\mu_n\geq h(\mu)+\beta_n\int f{\rm d}\mu$.
Dividing both sides by $\beta_n$ and letting $n\to\infty$ yields
$\int f{\rm d}\mu_n\to\Lambda_\sigma(f)$.
Hence $\mu$ is $f$-maximizing.

 Since the entropy map is upper semicontinuous, we have $\limsup_{n\to\infty} h(\mu_n)\leq h(\mu).$
If there were $\nu\in M_{\rm max}(f)$ such that $h(\mu)<h(\nu)$, then
for all sufficiently large $n$ we would have
\[h(\nu)+\beta_n\int f{\rm d}\nu=h(\nu)+\beta_n\int f{\rm d}\mu>h(\mu_n)+\beta_n\int f{\rm d}\mu_n,\]
in contradiction to the assumption that $\mu_n$ is an equilibrium state for the potential $\beta_nf$.
 Finally we claim $\liminf_{n\to\infty} h(\mu_n)\geq h(\mu)$, for otherwise we would reach a contradiction by applying the same argument to a subsequence of $(\mu_n)_{n\in\mathbb N}$.
\end{proof}
\subsection{Approximations of ergodic measures}
\label{construct-sec}

We are in position to state a proposition we have been leading up to.
\begin{prop}\label{hyp-lem2}
Let $\Sigma$ be a subshift with positive topological entropy that has a language decomposition $\mathcal L(\Sigma)=\mathcal C^p\mathcal G\mathcal C^s$ with $h(\mathcal C^p\cup\mathcal C^s)=H$
for some $\Hc\in[0,h_{\rm top}(\Sigma))$. 
For each $\mu\in M(\Sigma,\sigma)$ that is ergodic and satisfies $h(\mu)>\Hc$, the following statements hold:

\begin{itemize}
\item[(a)] For any open subset $U$ of $M(\Sigma,\sigma)$ that contains $\mu$, 
there exists a Lipschitz continuous function $f\colon\Sigma\to\mathbb R$ such that:
\begin{itemize}

\item[(i)] For any $\beta\geq 0$, there exists a unique equilibrium state for the potential $\beta f$, which has the weak Gibbs property, denoted by $\mu_{\beta f}$.

\item[(ii)] The map $\beta\in[0,\infty)\mapsto\mu_{\beta f}\in M(\Sigma,\sigma)$ is continuous and injective.

\item[(iii)] For all sufficiently large $\beta>0$ we have
$\mu_{\beta f}\in U$, and there exists $H_0>H$ such that $\lim_{\beta\to\infty}h(\mu_{\beta f})=H_0$.
\end{itemize}
\item[(b)] 
For any $\delta>0$ and any open subset $V$ of $M(\Sigma,\sigma)$ that contains $\mu$, there exists an ergodic measure $\nu\in V$ such that $h(\nu)<\delta$.
\end{itemize}
\end{prop}

Proposition~\ref{hyp-lem2} provides two different kinds of approximations of ergodic measures with sufficiently large entropy.
In part (a), the measure $\mu$ is approximated in the weak* topology by a ground state with positive entropy in the zero temperature limit. In part (b), $\mu$ is approximated in the weak* topology by measures with arbitrarily small entropy. From Proposition~\ref{hyp-lem2}, we deduce the abundance of paths of ergodic measures with high complexity and a statement on 
the density of ergodic measures with arbitrarily small entropy.
 \begin{cor}\label{join-lem}
Let $\Sigma$ be a subshift that satisfies $h_{\rm spec}^\bot(\Sigma)<h_{\rm top}(\Sigma)$. If $h_{\rm spec}^\bot(\Sigma)\leq H<h_{\rm top}(\Sigma)$, then for any $\mu\in M(\Sigma,\sigma)$ that is ergodic and satisfies $h(\mu)>H$, and for any open subset $U$ of $M(\Sigma,\sigma)$ that contains $\mu$, there exists a homeomorphism $t\in[0,1]\mapsto\nu_t\in U$ onto its image such that for any $t\in [0,1]$, $\nu_t$ is ergodic, fully supported and satisfies $h(\nu_t)>H$.
\end{cor}
\begin{proof}
Measures with the weak Gibbs property \eqref{w-Gibbs} are fully supported. Hence, Corollary~\ref{join-lem}
follows from part (a) of Proposition~\ref{hyp-lem2}.\end{proof}

\begin{cor}\label{join-lem-cor}Let $\Sigma$ be a subshift that satisfies $h_{\rm spec}^\bot(\Sigma)<h_{\rm top}(\Sigma)$. If $h_{\rm spec}^\bot(\Sigma)\leq H<h_{\rm top}(\Sigma)$, then for any $\mu\in M(\Sigma,\sigma)$ that is ergodic and satisfies $h(\mu)>H$, any open subset 
$V$ of $M(\Sigma,\sigma)$ that contains $\mu$ and any $\delta>0$, there exists a shift-invariant ergodic measure $\nu\in V$
satisfying $h(\nu)<\delta$. \end{cor}
\begin{proof}Follows from part (b) of Proposition~\ref{hyp-lem2}.\end{proof}
 
The rest of this subsection is entirely dedicated to a proof of Proposition~\ref{hyp-lem2}, split into the proof of (a) and that of (b).

\begin{proof}[Proof of Proposition~\ref{hyp-lem2}(a)]
The proof breaks into four steps.
First, for each $k\in\mathbb N$ we construct a collection $\mathcal Q_k\subset\mathcal L_k(\Sigma)$ by selecting words that are associated with the measure $\mu$. Next we construct a collection 
$\tilde{\mathcal Q}_k$ of words with word lengths approximately $k$ that has (W)-specification, by extracting good words in the core $\mathcal G$ from elements of $\mathcal Q_k$.
In the third step we glue the words in $\tilde{\mathcal Q}_k$ and construct a subshift in $\Sigma$ that approximates the measure $\mu$ in a particular sense.
Finally we verify (i) (ii) (iii).
\medskip

\noindent{\it Step~1: The selection of words.}
Let $\mu\in M(\Sigma,\sigma)$ be ergodic and satisfy $h(\mu)>\Hc$. 
Let $f_1,\ldots,f_m$ be an arbitrary finite collection of functions in $C(\Sigma)$.
Since the collection of $1$-cylinders in $\Sigma$ generates the Borel sigma-algebra of $\Sigma$, it follows
from Shannon-McMillan-Breiman's theorem and Birkhoff's ergodic theorem that  
for $\mu$-a.e. $x\in\Sigma$ we have
\[\lim_{k\to\infty}\frac{1}{k}\log\mu([x_1\cdots x_k])
=-h(\mu)\]
and
\[\lim_{k\to\infty}\frac{1}{k}S_kf_i(x)=\int f_i{\rm d}\mu \ \text{ for all $i\in\{1,\ldots,m\}$.}\]
We take 
$H_0\in(H,h(\mu))$, put \[h_0=h(\mu)-H_0,\]
and let
\begin{equation}\label{epsilon-def}\varepsilon\in\left(0,\min\left\{\frac{h_0}{18},\frac{h(\mu)-H}{(2+10H_0/h_0)}\right\}\right).\end{equation}
For each $x\in\Sigma$ for which the above two equalities hold, let $N(x)\in\mathbb N$
denote the minimal integer such that
\[\sup_{k\geq N(x)}\left|\frac{1}{k}\log\mu([x_1\cdots x_k])
+h(\mu)\right|<\varepsilon,\]
and
\[\sup_{k\geq N(x)}\max_{i\in\{1,\ldots,m\}}\left|\frac{1}{k}S_kf_i(x)-\int f_i{\rm d}\mu\right|<\frac{\varepsilon}{3}.\]
There exist a Borel set $\Sigma_0\subset\Sigma$ and
$N_0\in\mathbb N$ 
such that $0<\mu(\Sigma_0)<1$
and $N(x)=N_0$ for all $x\in \Sigma_0$. For $k\in\mathbb N$ we set
\begin{equation}\label{Hk}\mathcal Q_k=\{w\in\mathcal L_{k}(\Sigma)\colon[w]\cap \Sigma_0\neq\emptyset\}.\end{equation}
If $k\geq N_0$ is sufficiently large, then
for each $w\in\mathcal Q_k$ we have
\begin{equation}\label{size}
e^{-k(h(\mu)+\varepsilon)}\leq
\mu([w])\leq e^{-k(h(\mu)-\varepsilon)},\end{equation}
and
\begin{equation}\label{meas-eq1}\max_{x\in [w]}\max_{i\in\{1,\ldots,m\}}\left|\frac{1}{k}S_kf_i(x)- \int f_i{\rm d}\mu\right|<\frac{\varepsilon}{2}.
\end{equation}
Moreover, from \eqref{size} we have
\begin{equation}\label{meas-eq2}\mu(\Sigma_0)e^{(h(\mu)-\varepsilon)k}
\leq \#\mathcal Q_k\leq e^{(h(\mu)+\varepsilon)k}
.\end{equation}
\medskip

\noindent{\it Step~2: Extracting good words.} 
For the collections $\mathcal C^p,\mathcal C^s\subset\mathcal L(\Sigma)$ in the assumption,
take $N_1\in\mathbb N$ such that
\begin{equation}\label{meas-eq0}\#(\mathcal C_{n}^p\cup\mathcal C_{n}^s)\leq
e^{H_0  n}\ \text{ for all }n\geq N_1.\end{equation}
Put
\[N_2=N_0+N_1.\]

For each $w\in\mathcal Q_k$, we fix once and for all a decomposition \[w=p(w)c(w)s(w),\ \
p(w)\in\mathcal C^p, c(w)\in \mathcal G, s(w)\in\mathcal C^s.\] Note that $p(w)$, $c(w)$, $s(w)$ can be the empty word.
For $k\in\mathbb N$ and $a,b\in\mathbb N\cup\{0\}$ satisfying $a+b\leq k$, define
\[\mathcal Q_{k}(a,b)=\{w\in\mathcal Q_k\colon |p(w)|=a,  |s(w)|=b
\}.\]
Note that $\mathcal Q_k(a,b)$ are pairwise disjoint disjoint.
We also define
\[\begin{split}\mathcal Q^p_{k}(a,b)&=
\{p(w)\colon w\in\mathcal Q_k,\ |p(w)|=a,\ |s(w)|=b \},\\
\mathcal Q_{k}^c(a,b)&=\{c(w)\colon w\in\mathcal Q_k,\ |p(w)|=a,\ |s(w)|=b \},\\
\mathcal Q_k^s(a,b)&=
\{s(w)\colon w\in\mathcal Q_k,\ |p(w)|=a,\ |s(w)|=b \}.\end{split}\]
Since $\mathcal L(\Sigma)=\mathcal C^p\mathcal G\mathcal C^s$, we have
\[\mathcal Q_{k}=\bigcup_{a,b\in\mathbb N\cup\{0\},a+b\leq k}\mathcal Q_{k}(a,b).
\]

The next lemma asserts that the collection of words in $\mathcal Q_k$ with relatively short prefixes and suffixes has a definite fraction in exponential scale.
\begin{lemma}\label{empty-lem-new}For all sufficiently large $k\geq N_2$ we have
\[\sum_{a=0}^{\lfloor5\varepsilon k/h_0\rfloor-1}\sum_{b=0}^{\lfloor4\varepsilon k/h_0\rfloor-1}\#\mathcal Q_{k}(a,b)\geq\frac{\mu(\Sigma_0)}{2}e^{(h(\mu)-\varepsilon)k}.\]
\end{lemma}
\begin{proof}
For $k\in\mathbb N$ and $a,b\in\mathbb N\cup\{0\}$ with $a+b\leq k$,
put
\[\mathcal Q_k(a,\cdot)=\{ w\in\mathcal Q_k\colon |p(w)|=a\} \ \text{ and }\  \mathcal Q_{k}(\cdot,b)=\{w\in\mathcal Q_k\colon |s(w)|=b\}.\]
In what follows we derive three intermediate inequalities on the cardinalities of these collections, and combine them to deduce the desired one.

First, by \eqref{meas-eq0}, for all $a\in\{ N_1,\ldots,k\}$ we have
\[\#\{p(w)\colon w\in\mathcal Q_k,\ |p(w)|=a\}\leq e^{H_0a}.\]
 From \eqref{size}, 
 for all $a\in\{N_0,\ldots,k-1\}$
 and all  
$w\in\mathcal Q_k$ such that $|p(w)|=a$, 
 we have 
 \[\#\{v\in\mathcal L_{k-a}(\Sigma)\colon p(w)v\in\mathcal Q_k\}\leq e^{(k-a)h(\mu)+\varepsilon(k+a)}\leq e^{(k-a)h(\mu)+2\varepsilon k}.\]
  Multiplying these two estimates yield
 \[\#\mathcal Q_k(a,\cdot)\leq e^{H_0a}e^{(k-a)h(\mu)+2k\varepsilon}\ \text{ for all }a\in\{N_1,\ldots,k\},\]
and hence
\begin{equation}\label{cardi-eq1}\begin{split}\sum_{a=\lfloor5\varepsilon k/h_0\rfloor}^{k}\#\mathcal Q_k(a,\cdot)&\leq \sum_{a=\lfloor5\varepsilon k/h_0\rfloor}^{k}e^{H_0 a}e^{(k-a)h(\mu)+2\varepsilon k}\\
&\leq e^{(h(\mu)+2\varepsilon)k}
\sum_{a=\lfloor 5\varepsilon k/h_0\rfloor}^{k}
e^{-h_0a}
\leq e^{(h(\mu)-2\varepsilon)k},\end{split}\end{equation}
provided $k$ is sufficiently large.

Next, by \eqref{meas-eq2}, for all $b\in\{0,\ldots, k-N_0\}$ we have 
\[\#\{p(w)c(w)\colon w\in \mathcal Q_k,\ |s(w)|=b\}\leq\#\mathcal Q_{k-b}\leq e^{(h(\mu)+\varepsilon)(k-b) }.\]
By \eqref{meas-eq0}, for each $b\in\{N_1,\ldots,k\}$,
any element of the left set 
is concatenated to at most $e^{H_0b}$ elements of $\mathcal C_b^s$.
Hence
\[\#\mathcal Q_k(\cdot,b)\leq e^{(h(\mu)+\varepsilon)(k-b) }e^{H_0b}
\ \text{ for all }b\in\{N_1,\ldots,k-N_0\},\]
and so
\begin{equation}\label{cardi-eq2}\begin{split}\sum_{b=\lfloor4\varepsilon k/h_0\rfloor}^{k-N_0}\#\mathcal Q_{k}(\cdot,b)&\leq e^{(h(\mu)+\varepsilon)k }\sum_{b=\lfloor4\varepsilon k/h_0\rfloor}^ke^{-b(h(\mu)+\varepsilon-H_0  ) }\\
&\leq e^{(h(\mu)+\varepsilon)k }\sum_{b=\lfloor4\varepsilon k/h_0\rfloor}^ke^{-bh_0 }\leq
e^{(h(\mu)-2\varepsilon)k },\end{split}\end{equation}
provided $k$ is sufficiently large.

Finally, since $\Sigma$ is a subshift over the finite alphabet $A$ as in $\S\ref{nice}$, it is clear that
 \begin{equation}\label{cardi-eq3}\#\mathcal Q_{k}(\cdot,b)<(\#A)^{N_0}e^{H_0 k}\ \text{ for all }b\in\{k-N_0+1,\ldots,k\}.\end{equation}
Combining \eqref{cardi-eq1}, \eqref{cardi-eq2}, \eqref{cardi-eq3}
and the lower bound in \eqref{meas-eq2},
for all sufficiently large $k\geq N_2$
we obtain
\[\#\mathcal Q_k-\sum_{a=\lfloor5\varepsilon k/h_0\rfloor}^{k}\#\mathcal Q_{k}(a,\cdot)-\sum_{b=\lfloor4\varepsilon k/h_0\rfloor}^{k}
\#\mathcal Q_{k}(\cdot,b)\geq\frac{\mu(\Sigma_0)}{2}e^{(h(\mu)-\varepsilon)k},\]
which implies the desired inequality.
\end{proof}

From Lemma~\ref{empty-lem-new}, for all sufficiently large $k\geq N_2$
there exist $a_*$, $b_*\in\mathbb N\cup\{0\}$
such that 
 $a_*+b_*\leq9\varepsilon k/h_0$ and
\begin{equation}\label{low-h}\#\mathcal Q_{k}(a_*,b_*)\geq\frac{1}{(9\varepsilon k/h_0+1)^2}e^{(h(\mu)-\varepsilon)k}.\end{equation}
We set
\[\tilde{\mathcal Q}_k=\mathcal Q_{k}^c(a_*,b_*),\]
and
put
\[C=4+\frac{72}{h_0}\sup_{i\in\{1,\ldots,m\}}\left\|f_i\right\|_{C^0}.\]
\begin{lemma}\label{summary}For all sufficiently large $k\geq N_2$ we have
 \begin{equation}\label{m-eq1}
\#\tilde{\mathcal Q}_k\geq e^{(h(\mu)-\varepsilon-10\varepsilon H_0/h_0)k},
\end{equation}
and 
\begin{equation}\label{m-eq2}\max_{ x\in [v]}\max_{i\in\{1,\ldots,m\}}\left|\frac{1}{|v|}S_{|v|}f_i(x)- \int f_i{\rm d}\mu\right| < C\varepsilon\ \text{ for all $v\in\tilde{\mathcal Q}_k$.}\end{equation}
\end{lemma}
\begin{proof}
Since
$a_*+b_*\leq 9\varepsilon k/h_0<k$, we have
$\#\mathcal Q_{k}^c(a_*,b_*)\geq\#\mathcal Q_{k}(a_*,b_*)(\#\mathcal C^p_{a_*}\#\mathcal C^s_{b_*})^{-1}$. Then
\eqref{m-eq1} follows from 
\eqref{meas-eq0} and \eqref{low-h} provided $k\geq N_2$ is sufficiently large.
Let $v\in\tilde{\mathcal Q}_k$.
There exist $u\in\mathcal C^p_{a_*}$,
$w\in\mathcal C^s_{b_*}$
such that $uvw\in\mathcal Q_{k}(a_*,b_*)$, and so $|v|=k-a_*-b_*$.
Take a point
$y\in [uvw]$. For all $i\in\{1,\ldots,m\}$ we have
\[|S_{|v|}f_i(\sigma^{a_*}y)-S_kf_i(y)|\leq(a_*+b_*)\|f_i\|_{C^0}\leq\frac{9\varepsilon k}{h_0 }\|f_i\|_{C^0}.\]
By \eqref{meas-eq1}, we have
\[\frac{k-|v|}{|v|}
\int f_i{\rm d}\mu-\frac{k\varepsilon}{|v|}\leq\frac{S_kf_i(y)}{|v|}-\int f_i{\rm d}\mu\leq\frac{k-|v|}{|v|}
\int f_i{\rm d}\mu+\frac{k\varepsilon}{|v|}.\]
Therefore
\begin{equation}\label{combination1}\begin{split}\left|\frac{1}{|v|}S_{|v|}f_i(\sigma^{a_*}y)-\int f_i{\rm d}\mu\right|&<\frac{k\varepsilon}{|v|}+\frac{k-|v|}{|v|}\| f_i\|_{C^0}+\frac{9\varepsilon k}{h_0|v|}\|f_i\|_{C^0}\\
&\leq\frac{\varepsilon}{1-9\varepsilon /h_0}\left(1+\frac{18}{h_0}\|f_i\|_{C^0}\right)
< \frac{C\varepsilon}{2}.\end{split}\end{equation}
To deduce the second inequality we have used $k/|v|\leq k/(k-9\varepsilon k/h_0)=1/(1-9\varepsilon /h_0)$.
The last inequality follows from $\varepsilon<h_0/18.$
Meanwhile, by the uniform continuity of $f_i$, if $k$ is sufficiently large then for all $x\in[v]$ and all $i\in\{1,\ldots,m\}$ we have
\begin{equation}\label{combination2}\left|\frac{1}{|v|}S_{|v|}f_i(x)-\frac{1}{|v|}S_{|v|}f_i(\sigma^{a_*}y)\right|<\frac{C\varepsilon}{2}.\end{equation}
Combining \eqref{combination1} and \eqref{combination2} we obtain \eqref{m-eq2}.
\end{proof}

\medskip

\medskip

\noindent {\it Step~3: The construction of a subshift in $\Sigma$.}
Let $t\geq0$ be a gap size for the (W)-specification of $\mathcal G^1$.
If $\Sigma$ is one-sided, then
define
\[\Gamma=\{v^{(1)}w^{(1)}v^{(2)}w^{(2)}v^{(3)}w^{(3)}\cdots\in\Sigma\colon
v^{(j)}\in\tilde{\mathcal Q}_k, |w^{(j)}|\leq t\  \text{ for all } j\in\mathbb N\}.\]
If $\Sigma$ is two-sided, then
define
\[\Gamma=\{x=\cdots w^{(-1)}v^{(0)}w^{(0)}v^{(1)}\cdots\in\Sigma\colon 
v^{(j)}\in\tilde{\mathcal Q}_k, |w^{(j)}|\leq t\  \text{ for all } j\in\mathbb Z, [x_0]\supset[v_0]\}.\]
In other words, $\Gamma$ is the set of points in $\Sigma$ whose symbol sequences are alternate concatenations of elements of $\tilde{\mathcal Q}_k$ and that of $\mathcal L(\Sigma)$ with word lengths not exceeding $t$. The set
\[\tilde\Gamma=\overline{\bigcup_{n=0}^\infty\sigma^{n}\Gamma}\] is a subshift in $\Sigma$ over the same alphabet $A$.
Since $0<\mu(\Sigma_0)<1$, $\tilde\Gamma$ is a proper subset of $\Sigma$.
Define a function $f\colon\Sigma\to\mathbb R$ by \[f(x)=-\min\{d(x,y)\colon y\in\tilde\Gamma\},\]
where $d$ denotes the Hamming metric on $\Sigma$ (see $\S\ref{nice}$).
\begin{lemma}\label{Lipschitz-lem}$f$ is Lipschitz continuous.\end{lemma}
\begin{proof}Let $x,y$ be distinct points in $\Sigma$.
Take $i_0,i_1,i_2\in\mathbb N$ such that
$d(x,y)=2^{-i_0}$, $f(x)=-2^{-i_1}$, $f(y)=-2^{-i_2}$.
If $i_1\leq i_0$ or $i_2\leq i_0$, then $i_1=i_2$ and so
$f(x)-f(y)=0$. Otherwise, $|f(x)-f(y)|= |2^{-i_1}-2^{-i_2}|< 2^{-i_0+1}=2d(x,y)$.
\end{proof}
Lemma~\ref{Lipschitz-lem} implies
\begin{equation}\label{distortion}\sup_{n\geq1}\max_{w\in\mathcal L_n(\Sigma)}\max_{x,y\in[w]}(S_nf(x)-S_nf(y))<\infty.\end{equation}
Clearly we have 
$\Lambda_\sigma(f)=0$,
and hence
any $f$-maximizing measure belongs to $M(\tilde\Gamma,\sigma|_{\tilde\Gamma})$.

Let $x\in\Gamma$. We have
 $x=v^{(1)}w^{(1)}v^{(2)}w^{(2)}\cdots$, or
$x=\cdots w^{(-1)}v^{(0)}w^{(0)}v^{(1)}\cdots$
and $[x_0]\supset [v_0]$ according as $\Sigma$ is one-sided or two-sided, where 
$v^{(j)}\in\tilde{\mathcal Q}_k$ and $|w^{(j)}|\leq t$.
From \eqref{m-eq2}, for all $j$ 
we have
\[\begin{split}\max_{i\in\{1,\ldots,m\}}\left|\int S_{|v^{(j)}w^{(j)}|}f_i(x)-|v^{(j)}w^{(j)}|\int f_i{\rm d}\mu\right|&<C|v^{(j)}|\varepsilon+2|w^{(j)}|\|f_i\|_{C^0}\\&<2C|v^{(j)}w^{(j)}|\varepsilon.\end{split}\]
The last inequality holds for all sufficiently large $k$. It follows that there exists $N_3\geq1$ that is independent of $x\in\Gamma$ such that
for all $n\geq N_3$ we have
\[\left|\frac{1}{n}S_nf_i(x)-\int f_i{\rm d}\mu\right|<3C\varepsilon.\]
Passing to the limit, we obtain the same estimate for all $x\in\tilde\Gamma$ and all sufficiently large $n\geq N_3$.
Consequently, for any measure $\nu\in M(\tilde\Gamma,\sigma|_{\tilde\Gamma})$ we have
\[\max_{i\in\{1,\ldots,m\}}\left|\int f_i{\rm d}\nu-\int f_i{\rm d}\mu\right|\leq3C\varepsilon.\]
Since the collection $f_1,\ldots,f_m$ of functions in $C(\Sigma)$ is arbitrary and $\varepsilon>0$ can be chosen arbitrarily small, one can choose them so that  $M(\tilde\Gamma,\sigma|_{\tilde\Gamma})$ is contained in an arbitrarily small neighborhood of $\mu$ in $M(\Sigma,\sigma)$. 
\medskip

\noindent{\it Step~4: The verification of  (i) (ii) (iii).}
If $k$ is sufficiently large, then \eqref{m-eq1} implies
\begin{equation}\label{meas-eq3}h_{\rm top}(\tilde\Gamma)> h(\mu)-2\varepsilon-10\varepsilon H_0/h_0>\Hc.\end{equation}

Let $\nu_0$ be an ergodic measure of maximal entropy for the subshift $\tilde\Gamma$. Note that
\[\limsup_{n\to\infty}\sup_{x\in\Sigma}\frac{1}{n}S_n f(x)=\sup_{x\in\Sigma}\limsup_{n\to\infty}\frac{1}{n}S_nf(x)=\Lambda_\sigma(f)=0.\]
Combining this with \eqref{meas-eq3}, for all $\beta\geq0$ we get
\[\begin{split}h(\mathcal C^p\cup\mathcal C^s)+\beta\cdot\limsup_{n\to\infty}\sup_{x\in\Sigma}\frac{1}{n}S_nf(x)
&=H<h_{\rm top}(\tilde\Gamma)\\
&=h(\nu_0)=h(\nu_0)+\int \beta f{\rm d}\nu_0\leq P(\beta f).\end{split}\]
This estimate implies 
\[\sum_{n=1}^\infty\Lambda_n(\mathcal C^p\cup\mathcal C^s,\beta f)e^{-P(\beta f)n}<\infty.\]
By Theorem~\ref{CT-thm}, the equilibrium state for the potential $\beta f$ is unique, denoted by $\mu_{\beta f}$, and has the weak Gibbs property \eqref{w-Gibbs}. This verifies (i).

The continuity of the map $\beta\in[0,\infty)\mapsto\mu_{\beta f}\in M(\Sigma,\sigma)$ follows from the next general lemma. \begin{lemma}\label{continuity-lem}
Let $X$ be a subshift,
 let $\varphi\in C(X)$ and suppose that the equilibrium state for the potential $\varphi$ is unique, denoted by $\lambda_\varphi$.
  Let $(\lambda_n)_{n\in\mathbb N}$ be a sequence in $M(X,\sigma)$ and let $(\varphi_n)_{n\in\mathbb N}$ be a sequence in $C(X)$ such that:

  \begin{itemize}
 \item[(a)] $\|\varphi_n-\varphi\|_{C^0}\to0$ as $n\to\infty$.

\item[(b)] For each $n\in\mathbb N$, $\lambda_n$ is an equilibrium state for the potential $\varphi_n$. 
  \end{itemize}
  Then
 $\lambda_n\to\lambda_{\varphi}$ as $n\to\infty$ in the weak* topology of  $M(X,\sigma)$.
 \end{lemma}
 \begin{proof}
 Let $\lambda$ denote any limit point of $(\lambda_n)_{n\in\mathbb N}$ in the weak* topology of $M(X,\sigma)$. Since the pressure continuously depends on the potential and the entropy function is upper semicontinuous, the map $\lambda_n\in M(X,\sigma)\mapsto h(\lambda_n)+\int\varphi_n{\rm d}\lambda_n$ is upper 
 semicontinuous. So, we have
 \[P(\varphi)=\lim_{n\to\infty}\left(h(\lambda_n)+\int\varphi_n{\rm d}\lambda_n\right)\leq h(\lambda)+\int\varphi{\rm d}\lambda,\]
 and therefore $\lambda=\lambda_\varphi$, which completes the proof.
 \end{proof}To prove the injectivity of the map $\beta\in[0,\infty)\mapsto\mu_{\beta f}\in M(\Sigma,\sigma)$,
let $0\leq \beta_1<\beta_2$ and suppose $\mu_{\beta_1f}=\mu_{\beta_2f}$ by contradiction.
Since $f$ is non-positive and non-constant, we have $P(\beta_1f)>P(\beta_2f)$.
By the weak Gibbs property \eqref{w-Gibbs} of $\mu_{\beta_1f}$, $\mu_{\beta_2f}$
and \eqref{distortion},
there exists a constant $K\geq1$ such that for all $n\geq1$,
 $w\in\mathcal G\cap\mathcal L_n(\Sigma)$ and $x\in[w]$ we have
\[K^{-1}\leq \frac{\exp(-P(\beta_1f)n+\beta_1S_nf(x))}{\exp(-P(\beta_2f)n+\beta_2S_nf(x))}\leq K.\]
This is equivalent to
\begin{equation}\label{meas-eq30}|(\beta_1-\beta_2)S_nf(x)-n(P(\beta_1f)-P(\beta_2f))|\leq C,\end{equation}
for some constant $C>0$ independent of $n$, $w$, $x$. 
Since $\nu_0\in M(\tilde\Gamma,\sigma|_{\tilde\Gamma})$ is ergodic, has positive entropy and satisfies 
$\int f{\rm d}\nu_0=0$, in the same way as the construction of the collection of words $\tilde{\mathcal Q}_k=\tilde{\mathcal Q}_k(\mu)$ in \eqref{Hk},
for any $\delta>0$ one can construct  $\tilde{\mathcal Q}_k=\tilde{\mathcal Q}_k(\nu_0)$ to conclude that
for all sufficiently large $n\geq1$ there exists  $w\in\mathcal G\cap \mathcal L_n(\Sigma)$ such that  
$\sup_{x\in[w]}|(1/n)S_nf(x)|<\delta$.
This and \eqref{meas-eq30} together imply
$P(\beta_1f)=P(\beta_2f)$, a contradiction.
This verifies (ii).

By Lemma~\ref{zero-t-lem},
any ground state for the potential $f$ belongs to $M_{\rm max}(f)$, and so belongs to $M(\tilde\Gamma,\sigma|_{\tilde\Gamma })$.
Lemma~\ref{zero-t-lem} also yields
 $\lim_{\beta\to\infty}h(\mu_{\beta f})= h(\nu_0)=h_{\rm top}(\tilde\Gamma)>\Hc$. This together with the consequence of the last paragraph in Step~3 
 verifies (iii).
The proof of Proposition~\ref{hyp-lem2}(a) is complete.
 \end{proof}

\begin{proof}[Proof of Proposition~\ref{hyp-lem2}(b)]
Continuing from the proof of Proposition~\ref{hyp-lem2}(a),
let $\mu\in M(\Sigma,\sigma)$ be ergodic and satisfy $h(\mu)>\Hc$.
Pick a word $v$ from the collection $\tilde{\mathcal Q}_k=\mathcal Q_k(\mu)$ in \eqref{Hk}.
If $\Sigma$ is one-sided, define
\[\Omega=\{vw^{(1)}vw^{(2)}vw^{(3)}\cdots\in\Sigma\colon|w^{(j)}|\leq t\  \text{ for all } j\geq1\}.\]
If $\Sigma$ is two-sided, define
\[\Omega=\{x=\cdots vw^{(-2)}vw^{(-1)}vw^{(0)}vw^{(1)}v\cdots\in\Sigma\colon|w^{(j)}|\leq t\  \text{ for all }j\in\mathbb Z, [x_0]\supset [v]\}.\]
In other words, $\Omega$ is the set of points in $\Sigma$ whose symbol sequences are alternate concatenations of $v$ and elements of $\mathcal L(\Sigma)$ with word lengths not exceeding $t$. The set
\[\tilde\Omega=\overline{\bigcup_{n=0}^\infty\sigma^{n}\Omega}\] is a subshift in $\Sigma$ over the same alphabet $A$. 
Since the choice of the collection $\{f_1,\ldots,f_m\}\subset C(\Sigma)$ in Step~$1$ is arbitrary and $\varepsilon>0$ can be chosen arbitrarily small, one can choose them so that $M(\tilde\Omega,\sigma|_{\tilde\Omega})$ is contained in an arbitrarily small neighborhood of $\mu$ in $M(\Sigma,\sigma)$.
Since the gap size $t$ for the specification of $\mathcal G^1$ is independent of $k$, we have
$h_{\rm top}(\tilde\Omega)\to0$ as $k\to\infty$. From this and the variational principle for the subshift $\tilde\Omega$, it follows that the neighborhood of $\mu$ contains a shift-invariant ergodic measure with arbitrarily small entropy.
The proof of Proposition~\ref{hyp-lem2}(b) is complete.\end{proof}

\section{On the proofs of Theorems A and B}
We are almost ready to proceed to proving the main results on ergodic optimization.
After recalling some functional analytic ingredients
in $\S\ref{FA}$, we 
 complete the proof of Theorem~A in $\S\ref{pfthm-a}$.
 In $\S\ref{pf-newsec}$ we complete the proof of Theorem~B.
\subsection{Functional analysis}\label{FA}
The proof of part (b) of Theorem~A relies on the result of
Israel \cite[Section~V]{Isr79}, who proved an approximation theorem about tangent functionals to convex functions, and used it  for lattice systems in statistical mechanics to prove the existence of a dense set of continuous interactions for which there exist uncountably many ergodic equilibrium states. 
Below we recall his result, and some other auxiliary ones.

For a Banach space $V$ with a norm $\|\cdot\|$,
let $V^*$ denote 
the set of real-valued bounded linear functionals on $V$. For each $\mu\in V^*$
let $\|\mu\|$ denote the norm
\[\|\mu\|=\sup\left\{|\mu(f)|\colon f\in V,\ \|f\|=1\right\}.\] Let $\Lambda$, $\mu\in V^*$. We say:
\begin{itemize}

\item 
$\mu$ is {\it tangent} to $\Lambda$ at $f\in V$ if $\mu(f)\leq\Lambda(f+g)-\Lambda(f)$ holds for all $g\in V$.

\item 
 $\mu$ is {\it bounded} by $\Lambda$ if $\mu(f)\leq\Lambda(f)$  holds for all $f\in V$.

\item 
 $\Lambda$ is {\it convex} if 
 $\Lambda(tf+(1-t)g)\leq t\Lambda(f)+(1-t)\Lambda(g)$ holds for all $f$, $g\in V$ and $t\in[0,1]$.
\end{itemize}
\begin{thm}[\cite{Isr79}, Theorem~V.1.1]\label{bis}
Let $V$ be a Banach space and let $\Lambda\in V^*$ be convex and continuous.
For any $\mu_0\in V^*$ that is bounded by $\Lambda$, any $f_0\in V$ and any $\varepsilon>0$, there exist $\mu\in V^*$ and $f\in V$ such that $\mu$ is tangent to $\Lambda$ at $f$ and
\[\|\mu-\mu_0\|\leq\varepsilon\ \text{ and }\
\|f-f_0\|\leq\frac{1}{\varepsilon}(\Lambda(f_0)-\mu_0(f_0)+s),\]
where $s=\sup\{\mu_0(g)-\Lambda(g)\colon g\in V\}\leq0$.
\end{thm}
For a continuous map $T$ of a compact metric space $X$, 
the functional $\Lambda_T$ on $C(X)$ given by \eqref{alphaT}
is convex and continuous.
The next lemma characterizes maximizing measures in terms of $\Lambda_T$.
Recall that 
 $C(X)^*$ can be identified with the set of (finite) signed Borel measures on $X$ by
 Riesz's representation theorem.

\begin{lemma}[\cite{Bre08}, Lemma~2.3]\label{Bre}
Let $T$ be a continuous map of a compact metric space $X$ and
let $f\in C(X)$. Then $\mu\in C(X)^*$ is tangent to $\Lambda_T$ at $f$ if and only if
$\mu$ belongs to $M(X,T)$ and is $f$-maximizing.
\end{lemma}

If $T$ is a continuous map of $X$, then
for any $\mu\in M(X,T)$ there exists a unique Borel probability measure $b_{\mu}$ on $M(X,T)$ such that $b_\mu(M^{\rm e}(X,T))=1$ and
$\mu=\int_{M^{\rm e}(X,T)} \nu{\rm d}b_\mu(\nu)$,
where $M^{\rm e}(X,T)$ denotes the set of 
elements of $M(X,T)$ that are ergodic.
We call $b_\mu$ the {\it barycenter} of $\mu$.
Put
\[{\rm supp}(b_\mu)
=\bigcap\{F\colon\text{$F\subset M(X,T)$, closed, $b_\mu(F)=1$}\}.\]
Since $M(X,T)$ has a countable base, we have $b_\mu({\rm supp}(b_\mu))=1$.
\begin{lemma}\label{supp-lem}
Let $T$ be a continuous map of a 
compact metric space $X$. 

\begin{itemize}

\item[(a)]  If there exists a constant $C\geq0$ such that $h(\nu)\leq C$ for all $\nu\in {\rm supp}(b_\mu)$, then $h(\mu)\leq C$.

\item[(b)] Let $f\in C(X)$, $\mu\in M_{\rm max}(f)$. Then ${\rm supp}(b_\mu)$ is contained in $M_{\rm max}(f)$.
\end{itemize}
 \end{lemma}
 \begin{proof} Part (a) is a consequence of Jacobs' theorem \cite{Wal82}. For a proof of part (b), see
 \cite[Proposition~6]{Shi18} for example.\end{proof}
The next lemma asserts that the barycenter map $\mu\mapsto b_\mu$ from $M(X,T)$ to the set of Borel probability measures on $M(X,T)$ is isometric.
\begin{lemma}[\cite{Isr79}, Corollary~IV.4.2]\label{norm-lem}
Let $T$ be a continuous map of a 
compact metric space $X$.
For all $\mu_1$, $\mu_2\in M(X,T)$ 
we have \[\|b_{\mu_1}-b_{\mu_2}\|=\|\mu_1-\mu_2\|.\]
\end{lemma}


\subsection{Proof of Theorem~A}\label{pfthm-a}
Let $\Sigma$ be a subshift that satisfies $h_{\rm spec}^\bot(\Sigma)<h_{\rm top}(\Sigma)$. 
Suppose
$h_{\rm spec}^\bot(\Sigma)\leq H<h_{\rm top}(\Sigma)$. 
For each $n\in\mathbb N$, define
\[O_n=\left\{\mu\in \overline{M^{\rm e}(\Sigma, \sigma)}: 0\leq h(\mu)<H+\frac{1}{n}\right\},\]
and
\[
    U_n=\{f\in C(\Sigma)\colon \overline{M^{\rm e}(\Sigma,\sigma)}\cap M_{{\rm max}}(f)\subset O_n\}.\]
        Clearly, the set
$\mathscr{R}_H$ in Theorem~A is contained in $\bigcap_{n=1}^\infty U_n.$ Conversely, let 
$f\in\bigcap_{n= 1}^\infty U_n$.
The entropy of any ergodic measure in $M_{\rm max}(f)$ does not exceed $H$, and 
by Lemma~\ref{supp-lem}, the entropy of any non-ergodic measure in $M_{\rm max}(f)$ does not exceed $H$ either.
Hence we obtain 
$\mathscr{R}_H=\bigcap_{n= 1}^\infty U_n.$

By the upper semicontinuity of the entropy function, $O_n$ is an open subset of $\overline{M^{\rm e}(\Sigma, \sigma)}$.
From Corollary~\ref{join-lem-cor},  $O_n$ is a dense subset of $\overline{M^{\rm e}(\Sigma, \sigma)}$.
By \cite[Theorem~1.1]{Mor10}, $U_n$
is an open and dense subset of $C(\Sigma)$.
Therefore $\mathscr{R}_H$ is dense $G_\delta$ as required in part (a) of Theorem~A.

To prove part (b) of Theorem~A, pick $f_0\in C(\Sigma)\setminus \mathscr{R}_H$. 
By Lemma~\ref{supp-lem}(a), there exists an ergodic measure $\lambda\in M_{\rm max}(f_0)$ that satisfies $h(\lambda)>H$. 
Let $\varepsilon\in(0,1/2)$ and put
\[R=\left\{\nu\in M(\Sigma,\sigma)\colon  \int f_0{\rm d}\nu\geq\Lambda_\sigma(f_0)-\varepsilon^2\right\}.\] 
  By Corollary~\ref{join-lem} applied to $\lambda$, there exists a homeomorphism $t\in[0,1]\mapsto \nu_t\in R\cap M^{\rm e}(\Sigma,\sigma)$ onto its image
 such that $\nu_t$ is fully supported and satisfies $h(\nu_t)>H$ for all $t\in[0,1]$.
Let $m$ denote the Lebesgue measure on $[0,1]$
and define a Borel probability measure $\hat m$
on $M^{\rm e}(\Sigma,\sigma)$ by $\hat m(\cdot)=m\{t\in[0,1]\colon \nu_t\in\cdot\}$.
We have $\hat m(R\cap M^{\rm e}(\Sigma,\sigma))=1$. 
Put $\mu_0=\int_{M^{\rm e}(\Sigma,\sigma)}\nu{\rm d}\hat m(\nu)$. 
 Then $\mu_0$ belongs to $R$ and is bounded by $\Lambda_\sigma$
as an element of $C(\Sigma)^*$.
 Note that $b_{\mu_0}=\hat m$.
By Theorem~\ref{bis}, there exist $f\in C(\Sigma)$ and $\mu\in C(\Sigma)^*$ such that $\mu$ is tangent to $\Lambda_\sigma$ at $f$, and
\begin{equation}\label{bis-eq}\|\mu-\mu_0\|\leq\varepsilon\ \text{ and }\ \|f-f_0\|_{C^0}\leq\frac{1}{\varepsilon}\left(\Lambda_\sigma(f_0)-\int f_0{\rm d}\mu_0\right)\leq\varepsilon.\end{equation}
By Lemma~\ref{Bre}, $\mu$ belongs to $M(\Sigma,\sigma)$ and is $f$-maximizing.

Take an open subset $U$ of $M(\Sigma,\sigma)$ such that 
${\rm supp}(b_\mu)\subset U$
and $b_{\mu_0}(U\setminus{\rm supp}(b_\mu))<\varepsilon$.
Since $M^{\rm e}(\Sigma,\sigma)$ is a metric space, it is a normal space. 
By Urysohn's lemma, there exists a function $g\in C(M(\Sigma,\sigma))$ such that $\|g\|_{C^0}=1$, 
$g\equiv0$ on $M(\Sigma,\sigma)\setminus U$ and
$g\equiv1$ on ${\rm supp}(b_\mu)$.
We have
\[\begin{split}b_{\mu_0}({\rm supp}(b_\mu))&>
b_{\mu_0}(U)-\varepsilon>b_{\mu_0}(g)-\varepsilon\geq b_\mu(g)-2\varepsilon\\
&\geq b_\mu({\rm supp}(b_\mu))-2\varepsilon=1-2\varepsilon>0.\end{split}\]
To deduce the third inequality, we have used
$\|b_{\mu}-b_{\mu_0}\|=\|\mu-\mu_0\|\leq\varepsilon$
from Lemma~\ref{norm-lem} and the first inequality in \eqref{bis-eq}. 
Since $b_{\mu_0}$ is non-atomic, it follows that 
the set $\{\nu_t\colon t\in[0,1],\ \nu_t\in {\rm supp}(b_\mu)\}$
contains uncountably many elements, which belong to $M_{\rm max}(f)$ by Lemma~\ref{supp-lem}.
Since $f_0\in C(\Sigma)$ and $\varepsilon\in(0,1/2)$ are arbitrary,
the proof of part (b) of Theorem~A is complete.\qed


\subsection{Proof of Theorem~B }\label{pf-newsec}
Let $\Sigma$ be a subshift that satisfies $h_{\rm spec}^\bot(\Sigma)=0$, 
and suppose that ergodic measures on $\Sigma$ are entropy dense. Recall that \[\mathscr{R}_0=\{f\in C(\Sigma)\colon h(\mu)=0
\text{ for all $\mu\in M_{\rm max}(f)$}\},\] and let $f\in\mathscr{R}_0$. 
There exists an ergodic measure $\mu\in M_{\rm max}(f)$ with $h(\mu)=0$. From the entropy density and the affinity of the entropy function, 
$\mu$ is approximated in the weak* topology by ergodic measures with positive entropy.
Taking these approximating measures and repeating the argument in the proof of part (b) of Theorem~A in $\S\ref{pfthm-a}$, one can show that $f$ is in the closure of functions for which there exists a maximizing measure with positive entropy. Hence,  $\mathscr{R}_0$ has empty interior. 
Then the desired statement is a direct consequence of part(b) of Theorem~A.\qed

\section{Examples and applications}
In this section we provide examples of non-Markov symbolic dynamical systems to which Theorem~A or Theorem~B applies.

\subsection{Piecewise monotonic maps}\label{L-decomp}
We say $T\colon [0,1]\to [0,1]$
is a {\it piecewise monotonic map} if
there exist 
finitely many non-degenerate closed subintervals $I_1,\ldots,I_k$ of $[0,1]$ with disjoint interiors such that $\bigcup_{i=1}^k I_i=[0,1]$, and the restriction of $T$ to the interior ${\rm int} I_i$ of  
$I_i$ is strictly monotone and continuous for each $i\in\{1,\ldots,k\}$.

The intervals $I_i$ are called {\it monotonic pieces} of $T$.
We say $T$ is {\it transitive} if there is $x\in [0,1]$ such that $\{T^n(x)\colon n\ge 0\}$ is dense in $[0,1]$.
Given a transitive piecewise monotonic map $T$ with $k$ monotonic pieces,
let $X_T=\bigcap_{n=0}^\infty T^{-n}(\bigcup_{j=1}^k {\rm int}I_j)$ and
define $\pi\colon X_T\to \{1,\ldots,k\}^{\mathbb{N}}$ by
$x\in \bigcap_{n=0}^\infty T^{-n}({\rm int}I_{(\pi(x))_n})$.
The closure of $\pi(X_T)$ in the full shift space $\{1,\ldots,k\}^{\mathbb{N}}$ is a subshift, called the {\it coding space} of $T$ and denoted by $\Sigma_T$. 
If $T$ has positive topological entropy, then $h_{\rm top}(\Sigma_T)>0$.

Continuous piecewise monotonic maps have Bowen's specification property (see e.g., \cite{Buz97}), while
piecewise monotonic maps with discontinuities rarely have that. 
Typical examples of such maps 
are as follows:

\begin{itemize}
\item (the $(\alpha,\beta)$-transformation) 
$T_{\alpha,\beta}(x)=\beta x+\alpha-\lfloor\beta x+\alpha\rfloor$
 ($\alpha\in[0,1)$, $\beta>1$); 
\item (the $(-\beta)$-transformation) 
$T_{-\beta}(x)=-\beta x+\lfloor\beta x\rfloor+1$ ($\beta>1$).
\end{itemize}
The map $T_{0,\beta}$ is known 
as the $\beta$-transformation. The set of $\beta>1$ for which the corresponding coding space $\Sigma_{T_{0,\beta}}$ has 
Bowen's specification property is of zero Lebesgue measure \cite{Sch97}. 
For each fixed $\alpha\in(0,1)$,
the set of $\beta>1$ for which $\Sigma_{T_{\alpha,\beta}}$ has Bowen's specification property is of zero Lebesgue measure \cite{Buz97}.
The set of $\beta>1$ for which $\Sigma_{T_{-\beta}}$ has Bowen's specification property is of zero Lebesgue measure \cite{Buz97}.
For both families of transformations, the sets of parameters corresponding to maps having Bowen's specification property are not negligible in terms of Hausdorff dimension \cite{HHY17,OS24,Sch97}.

\begin{prop}
\label{CT-dec}
If $T\colon[0,1]\to[0,1]$ is a transitive piecewise monotonic map with positive topological entropy,
then $h_{\rm spec}^\bot(\Sigma_T)=0$.
\end{prop}


The original idea in our proof of Proposition~\ref{CT-dec} is due to Climenhaga 
(\url{https://www.math.uh.edu/~climenha/doc/marseille-specification.pdf}).
We will use a Markov diagram \cite{H2}, a  directed graph whose vertices are closed subsets of the coding space.
For the rest of this subsection,
let $T$ be a piecewise monotonic map as in Proposition~\ref{CT-dec} with $k$ monotonic pieces.
Let $C$ be a non-empty closed subset of the coding space $\Sigma_T\subset\{1,\ldots,k\}^{\mathbb N}$ such that $C\subset [i]$ holds for some $i\in\{1,\ldots,k\}$ (recall the notation for cylinders in $\S\ref{TF}$).
We say a non-empty closed subset $D$ of $\Sigma_T$ is a {\it successor} of $C$ if there exists $j\in\{1,\ldots,k\}$ such that
$D=[j]\cap\sigma C$. If $D$ is a successor of $C$, we write $C\to D$.
We set
$\mathcal{D}_0=\{[1],\ldots,[k]\}$,
and define $\mathcal D_1$, $\mathcal D_2,\ldots$ 
by the recursion formula
\[\mathcal{D}_{n+1}=\mathcal{D}_n\cup \{D\colon D\text{ is a successor of some }C\in\mathcal{D}_n\}.\]
We set
\[\mathcal{D}=\bigcup_{n= 0}^\infty\mathcal{D}_n.\]
The directed graph $(\mathcal D,\to)$ is called a {\it Markov diagram}
associated with $\Sigma_T$.

For a set $\mathcal C\subset\mathcal D$,
 let $\Sigma_\mathcal C$ denote the set of 
 paths 
 $C_0\to C_1\to \cdots$
 that contains infinitely many edges and all whose vertices are contained in $\mathcal C$.
We define a map
$\Psi\colon \Sigma_{\mathcal{D}}\to\{1,\ldots,k\}^{\mathbb N}$ as follows.
For each vertex $C_0$ or each path $C_0\to \cdots \to C_{n-1}$ in the diagram $(\mathcal D,\to)$, define
\[(C_0\cdots C_{n-1})=x_0\cdots x_{n-1}\in\{1,\ldots,k\}^n,\]
where
$x_i\in\{1,\ldots,k\}$ and $C_i\subset [x_i]$ for $0\leq i\leq n-1$.
For $(C_n)_{n=0}^\infty\in \{1,\ldots,k\}^\mathbb N$
define 
\[ \Psi((C_n)_{n=0}^\infty)\in\bigcap_{n=1}^\infty[(C_0\cdots C_{n-1})].\]
Note that $\Sigma_\mathcal D$ is a Markov shift over the infinite alphabet $\mathcal D$, $\Psi(\Sigma_\mathcal D)=\Sigma_T$, and 
$\Psi$ semiconjugates
the left shift on $\Sigma_\mathcal D$ to
$\sigma\colon\Sigma_T\to\Sigma_T$.

We say a non-empty subset $\mathcal C$ of $\mathcal D$ is {\it connected} if
 for any pair $(C,C')$ of vertices in $\mathcal C$ there is a path that 
 joins $C$, $C'$ and all whose vertices are contained in $\mathcal C$.
We say $\mathcal C\subset\mathcal D$ is an {\it irreducible component} if it is connected and not strictly contained in a subset of $\mathcal D$ that is connected.
The next lemma allows us to reduce our consideration to a `closed' irreducible component.
\begin{lemma}\label{irred-lem}There exists an irreducible component
$\mathcal{C}\subset\mathcal{D}$
such that:
\begin{itemize}
\item 
$\Psi(\Sigma_{\mathcal{C}})=\Sigma_T$.

\item
$C\in\mathcal{C}$ and $C\rightarrow D$
imply $D\in\mathcal{C}$.
\end{itemize}
\end{lemma}
\begin{proof}
Since
$T$ is transitive and has positive topological entropy, the desired conclusion
follows from \cite[Theorem~14]{H}.
\end{proof}

\begin{proof}[Proof of Proposition~\ref{CT-dec}]\label{last}Let $\mathcal C\subset\mathcal D$ be an irreducible component for which the conclusion of Lemma~\ref{irred-lem} holds.
It follows from \cite[Theorem~10]{H} (see also \cite[p.226]{HR98})
that there exists a finite subset
$\mathcal{F}$ of $\mathcal{C}$ such that
\begin{equation}\label{l-dec-eq1}\Psi(\{(C_n)_{n=0}^\infty\in\Sigma_{\mathcal{C}}\colon C_0\in\mathcal{F}\})
=\Psi(\Sigma_{\mathcal{C}})=\Sigma_T.\end{equation}
Let $H>0$, and let $N> (4/H)\log2$ be an integer such that
$\mathcal{F}\subset \mathcal{D}_N$. 
We now define collections
 $\mathcal C^p$, $\mathcal G$, $\mathcal C^s\subset \mathcal L(\Sigma_T)$ by
\[\begin{split}\mathcal C^p&=\{\emptyset\},\\
\mathcal G&=\bigcup_{n=1}^\infty
\{(C_0\cdots C_{n-1})\colon C_0\to\cdots\to C_{n-1},\ C_0,C_{n-1}\in
\mathcal{D}_N\cap\mathcal{C}\},\\
\mathcal C^s&=\bigcup_{n=1}^\infty \left\{\begin{tabular}{l}\vspace{1mm}
$\!\!\displaystyle{(C_0\cdots C_{n-1})\colon C_0\to\cdots \to C_{n-1},
\!\!}$\\
$\!\!\displaystyle{ C_0\in
\mathcal{D}_{N+1},\
C_i\in\mathcal{C}\setminus\mathcal{D}_N\text{ for }0\le i\le n-1}$\end{tabular}
\!\!\right\}.\end{split}\]

Let $w\in\mathcal{L}(\Sigma_T)\setminus\{\emptyset\}$.
By \eqref{l-dec-eq1}, there exists $(C_n)_{n=0}^\infty\in\Sigma_{\mathcal{C}}$
such that $\Psi((C_n)_{n=0}^\infty)\in[w]$ and
$C_0\in\mathcal{F}$.
Set $j_0=\max\{0\le i\le |w|-1\colon C_i\in\mathcal{D}_N\}$ and  $u=(C_0\cdots C_{j_0})$,
$v=(C_{j_0+1}\cdots C_{|w|-1})$.
It is easy to see that
$u\in\mathcal G$,
$v\in\mathcal C^s$ and
$w=uv$, which implies 
 $\mathcal{L}(\Sigma_T)\subset\mathcal C^p\mathcal G\mathcal C^s=\mathcal G\mathcal C^s$.
 The reverse inclusion is obvious.

For $M\in\mathbb N$, let $t_M\geq2$ denote the minimal integer such that any pair of vertices in $\mathcal{D}_{M+N}\cap \mathcal{C}$ is joined by a path 
that does not contain more than $t_M$ edges.
 Since $\mathcal{C}$ is an irreducible component and $\mathcal{D}_{M+N}$ is a finite set,
 $t_M$ is finite.
Let $u,v\in\mathcal G^M$.
By the above definition of $\mathcal G$, there exist two paths
$C_0\to\cdots \to C_{|u|-1}$ and $D_0\to\cdots\to D_{|v|-1}$
such that $C_{|u|-1}\in\mathcal{D}_{M+N}\cap \mathcal{C}$
and $D_0\in\mathcal{D}_{N}\cap\mathcal{C}$.
Hence there is a path
$C_{|u|-1}\to\cdots\to D_0$ 
not containing more than $t_M$ edges, which implies
that $\mathcal G^M$
has (W)-specification with a gap size $t_M$.

From the proof of
\cite[Corollary~1(i)]{H},
for all $q\ge 1$ and all
$C\in\mathcal{D}_{N+1}$, the number of paths starting at $C$, containing no more than $qN$ edges, and all whose vertices are contained in  $\mathcal{C}\setminus\mathcal{D}_N$
does not exceed $2^q$.
For $n=qN+r$, $q\in\mathbb N$, $r\in\{0,\ldots,N-1\}$
we have
\[\#\mathcal C^s_n \le
\#\mathcal C^s_{(q+1)N} \le
2^{q+1}\#\mathcal{D}_{N+1}.\]
For all sufficiently large $n\geq N$ we obtain
\[\frac{1}{n}\log\#\mathcal C^s_n
\le \frac{q+1}{n}\log 2+\frac{1}{n}\log
\#\mathcal{D}_{N+1}\le
\frac{2}{N}\log 2+\frac{H}{2}<H.\]
Since $H>0$ is arbitrary, $h_{\rm spec}^\bot(\Sigma_T)=0$ holds. 
\end{proof}

By
 Proposition~\ref{CT-dec},
if $T$ is a transitive piecewise monotonic interval map with positive topological entropy then
 $h_{\rm spec}^\bot(\Sigma_T)=0$. 
The entropy density for $\Sigma_T$ follows from
the density of shift-invariant
closed orbit measures in the space of ergodic measures \cite[Proposition~3.1]{ShiYam23}.
The latter condition is equivalent to the density of $T$-invariant closed orbit measures in the space of ergodic measures \cite[Theorem~A]{Y20}.
The density of closed orbit measures has been verified for a wide class of transitive piecewise monotonic maps with positive topological entropy (see e.g., \cite{Hof87,Hof88, HR98,ShiYam23, Sig76}). Hence,
Theorem~B applies to the coding spaces of these piecewise monotonic maps.
It has been conjectured that the density of closed orbit measures holds for any transitive piecewise monotonic map with positive topological entropy \cite{HR98}.



\subsection{Coded shifts}\label{coded-sec}
Let $A\subset\mathbb N$ be a finite set.
Let $A^0$ denote the singleton consisting of the empty word $\emptyset$, and put
$A^*=\bigcup_{n=0}^\infty A^n$. Let $\mathcal H
\subset A^*$ be a countable set containing $A^0$. 
The {\it coded shift} $\Sigma_{\mathcal H}$ generated by $\mathcal H$ is the two-sided subshift whose 
language is  
\[    \overline{\{w^{(1)} w^{(2)}\cdots w^{(k)}\colon w^{(1)},\ldots, w^{(k)}\in\mathcal H, k\in\mathbb N\}},
\]
where $\overline{\{\cdot\}}$ denotes the closure under the operation of passing to subwords.

The coded shift $\Sigma_\mathcal H$ has the following natural language decomposition $\mathcal{L}(\Sigma_\mathcal H)=\mathcal{C}^p\mathcal{G}\mathcal{C}^s$ for which  $\mathcal G^M$
has (W)-specification for all $M\in \mathbb{N}$:
\begin{equation}\label{n-dec}\begin{split}
     \mathcal{C}^p&=\bigcup_{k=1 }^\infty i_k^s(\mathcal H ),\\
    \mathcal{G}&=\bigcup_{k=1}^\infty\{w^{(1)}\cdots w^{(k)}\colon
    w^{(1)},\ldots,w^{(k)}\in \mathcal H  \},\\
    \mathcal{C}^s&=\bigcup_{k=1 }^\infty i_k^p(\mathcal H ),\\
\end{split}\end{equation}
where $i_k^s\colon \mathcal H\to \mathcal L_k(\Sigma_\mathcal H)\cup\{\emptyset\}$ (resp. $i_k^p\colon \mathcal H\to \mathcal L_k(\Sigma_\mathcal H)\cup\{\emptyset\}$)
is the map that extracts the last
(resp. first) $k$ symbols of a word in the generator $\mathcal H$, and
 satisfies $i_k^s(w)=\emptyset$ if $|w|<k$ (resp. $i_k^p(w)=\emptyset$ if $|w|<k$).
When $\mathcal H$ is a finite set, 
clearly $h(\mathcal C^p\cup\mathcal C^s)=0$ holds and $\Sigma_{\mathcal H}$ is a sofic shift.

Coded shifts with countably infinite generators are sources of interesting examples. 
Well-known examples are $S$-gap shifts \cite{LM} and the Dyck shift \cite{Kri74}. As shown below, there exist intrinsically ergodic coded shifts for which the entropy density fails. For more details on coded shifts, see \cite{CT12,KSW23}.
\subsubsection{$S$-gap shifts}
Let $S$ be an infinite subset of $\mathbb N\cup\{0\}$.
An {\it $S$-gap shift} \cite{LM} is a two-sided subshift on $\{1,2\}$ defined by the rule that the number of $2$'s between consecutive $1$'s is an integer in $S$.
To be more precise,
for $a\in A$ and $n\in\mathbb N$ let $a^n$ denote the $n$-fold concatenation of $a$ and put $a^0=\emptyset$, the empty word.
The $S$-gap shift is the coded shift generated by
$\{2^n1\colon n\in S\}\cup\{\emptyset\}$.
Note that the full shift $\{1,2\}^\mathbb Z$ is the $S$-gap shift with 
$S=\mathbb N\cup\{0\}$.
Generic $S$-gap shifts do not have Bowen's specification property
 \cite[$\S3.3$]{CT12}.
 The entropy of the $S$-gap shift is $\log\lambda$, where $\lambda$ is the unique solution to $1=\sum_{n\in S}x^{-n-1}$ (see \cite[Exercise~4.3.7]{LM}).
 The natural decomposition \eqref{n-dec} immediately gives
 $\#\mathcal C^p_n=1=\#\mathcal C^s_n$ for any $n\in\mathbb N$, and so
the obstruction entropy to specification is $0$. 
Hence, Theorem~A applies to any $S$-gap shift. Since
ergodic measures on any $S$-gap shift are entropy dense \cite{CTY17},
 Theorem~B applies to any $S$-gap shift.

 \subsubsection{Coded shifts without the entropy density}
Kucherenko et al. \cite[Example~1]{KSW23} considered coded shifts $\Sigma$ on two symbols
generated by $\{1^{i}2^{i}\colon i\in I\}\cup\{\emptyset\}$ 
where $I$ is an infinite subset of $\mathbb N$. 
The natural decomposition \eqref{n-dec}
gives
$h(\mathcal{C}^p\cup \mathcal{C}^s)=0$ and $0=h_{\rm spec}^\bot(\Sigma)<h_{\rm top}(\Sigma)$, and so Theorem~A applies.
Theorem~B does not apply since 
ergodic measures are not dense in $M(\Sigma, \sigma)$ and consequently the entropy density fails. Indeed, if both $1$ and $2$ appear in
$x\in\Sigma$ infinitely many times, then we have
\[\lim_{n\to\infty}\frac{1}{n}\#\{i\in\{0,\ldots,n-1\}\colon x_i=1\}=\frac{1}{2},\]
provided the limit exists.
 For $a\in\{1,2\}$
 define
$a^\infty\in \{1,2\}^\mathbb Z$ by
$(a^\infty)_i=a$ for all $i\in\mathbb Z$, 
 and for $x\in\Sigma$ let $\delta_x$ denote the unit point mass at $x$. From the above observation and Birkhoff's ergodic theorem, it follows that any non-ergodic measure of the form
$(1-t)\delta_{1^\infty}+t\delta_{2^\infty}$, $t\in(0,1)\setminus\{1/2\}$
 cannot be approximated by ergodic ones.

\subsection{Multidimensional $\beta$-transformations}\label{other}
Let $\Sigma$ be a subshift.
We say $w=w_1\cdots w_{|w|}\in\mathcal L(\Sigma)$ is a {\it left constraint} if there exists $v\in\mathcal L(\Sigma)$
such that $w_2\cdots w_{|w|}v\in\mathcal L(\Sigma)$ and $wv\notin\mathcal L(\Sigma)$. A {\it right constraint} is defined analogously. Let $\mathcal C^\ell$, $\mathcal C^{r}$ denote the collections of left constraints, right constraints respectively.
If $\Sigma$ is topologically transitive, then 
 $\mathcal C^\ell$, $\mathcal C^{r}$ form a complete list of obstructions to specification \cite[Theorem~1.6(1)]{C18}.
Combining \cite[Theorem~1.1]{C18},
\cite[Theorem~1.4]{C18},
\cite[Theorem~1.6]{C18} we obtain the following statements.
\begin{prop}\label{apply-prop} Let $\Sigma$ be a subshift.
\begin{itemize}
    \item[(a)]  Suppose $\Sigma$ is topologically transitive. If
$h(\mathcal C^\ell\cup\mathcal C^r)<h_{\rm top}(\Sigma)$,
then $h_{\rm spec}^\bot(\Sigma)<h_{\rm top}(\Sigma)$ holds and Theorem~A applies.
    \item[(b)]  Suppose $\Sigma$ is one-sided and topologically exact: for any $w\in\mathcal L(\Sigma)\setminus\{\emptyset\}$ there exists $N\in\mathbb N$
such that 
$\sigma^N[w]=\Sigma$.
If 
$h(\mathcal C^\ell)<h_{\rm top}(\Sigma)$,
then $h_{\rm spec}^\bot(\Sigma)<h_{\rm top}(\Sigma)$ holds and Theorem~A applies.
\end{itemize}\end{prop}


Part (b) of Proposition~\ref{apply-prop} applies to many of the piecewise affine transformations investigated by Buzzi  \cite{Buz97}. These are maps on $[0,1)^d$, $d\in\mathbb N$ given by $T(x)=Bx+b$ mod $\mathbb Z^d$, where $B\colon\mathbb R^d\to\mathbb R^d$ is an expanding linear map and $b\in\mathbb R^d$. The coding space $\Sigma_T$ of $T$ is defined as in the case of piecewise monotonic maps in $\S\ref{L-decomp}$, using the partition that consists of  maximal open subsets of $(0,1)^d$ on which $T(x)-Bx$ is constant. Buzzi showed that  $h(\mathcal C^\ell)<h_{\rm top}(\Sigma_T)$, and that $\Sigma_T$ is topologically exact if either (i) all eigenvalues of $B$ exceed $1+\sqrt{d}$ in absolute value,
or (ii) $B$, $b$ have all integer entries.
Since the maps with $d=1$ are almost the $\beta$-transformations,  the maps with $d\geq2$ are called {\it multidimensional $\beta$-transformations}.

For the coding spaces of multidimensional $\beta$-transformations, little is known when the obstruction entropy specification is $0$, or when the entropy density holds and Theorem~B applies.



\section{Subshifts with positive obstruction entropy to specification}

This last section is devoted to Theorem~C.
In $\S$\ref{fatten-sec} we introduce fat $S$-gap shifts by modifying the definition of $S$-gap shifts. In $\S$\ref{coded-positive} we complete the proof of Theorem~C by showing that any fat $S$-gap shift on $N$ symbols, $N\geq3$ has positive obstruction entropy to specification.
\subsection{Fat $S$-gap shifts}\label{fatten-sec}

Let $S$ be an infinite subset of $\mathbb{N}\cup\{0\}$.
A {\it fat $S$ gap shift} is a two-sided subshift on $\{1,\ldots,N\}$, $N\geq2$ defined by the rule that the word lengths of the words between consecutive $1$'s are integers in $S$.
More precisely, the fat $S$-gap shift $\Sigma_{S,N}$ on $N$ symbols is the coded shift generated by \begin{equation}\label{generator}\mathcal H=\{w1\colon w\in \{2,\ldots,N\}^n,\ n\in S\}\cup\{\emptyset\}.\end{equation}
Note that $\Sigma_{S,2}$ is nothing but an $S$-gap shift, whose obstruction entropy to specification is $0$.
  Note also that the full shift $\{1,\ldots,N\}^\mathbb Z$ is the $S$-gap shift on $N$ symbols with 
$S=\mathbb N\cup\{0\}$.

For all $S\subset\mathbb N\cup\{0\}$ and $N\geq2$
we clearly have $\Sigma_{S,N}\supset \{2,\ldots,N\}^{\mathbb Z}$, and thus $h_{\rm top}(\Sigma_{S,N})\geq \log (N-1).$ 
The next lemma asserts that this inequality is strict if $N\geq3$.

\begin{lemma}\label{top-low}For any infinite set $S\subset\mathbb N$ and any integer $N\geq3$, we have
  \[h_{\rm top}(\Sigma_{S,N})>\log (N-1).\]\end{lemma}

\begin{proof}
The original idea of the proof of this lemma is due to Climenhaga 
(\url{https://vaughnclimenhaga.wordpress.com/2014/09/08/entropy-of-s-gap-shifts/}). 
For ease of notation
we set $l_n=\#\mathcal{L}_n(\Sigma_{S,N})$.
The radius of convergence of
the power series
$F(x)=\sum_{n=1}^{\infty}l_nx^n$
equals $e^{-h_{\rm top}(\Sigma_{S,N})}$.
For $k,n\in\mathbb N$ we set
\[A_n^k=\#\{w^{(1)}\cdots w^{(k)}\colon
w^{(1)},\ldots,w^{(k)}\in\mathcal{H}\setminus\{\emptyset\},\ |w^{(1)}|+\cdots+
|w^{(k)}|=n\}.\]
Then it is not difficult to
prove the following properties:
\begin{itemize}
\item[(i)] $\sum_{k=1}^{\infty}A_n^k\le l_n$.
\item[(ii)] 
$A_n^1=(N-1)^{n-1}$ if $n-1\in S$ and $A_n^1=0$ otherwise. 
\item[(iii)] $A_n^{k+\ell}=\sum_{m=1}^{n-1}A_{n-m}^kA_m^\ell$ for all $k,\ell\in\mathbb N$.
\end{itemize}

We introduce a power series 
$F_k(x)= \sum_{n=1}^{\infty}A_n^kx^n$.
Then (iii) gives
$F_{k}(x)=F_{k-1}(x)F_1(x)=\cdots=F_1(x)^k$.
For all $x\ge 0$ we have
$$\sum_{k=1}^{\infty}F_1(x)^k
=\sum_{k=1}^{\infty}F_k(x)
=\sum_{n=1}^{\infty}\sum_{k=1}^{\infty}A_n^kx^n
\le F(x),$$
where the last inequality follows from
(i).
Let $x_0$ denote the unique positive solution to the equation
\[N-1=\sum_{n\in S}((N-1)x)^{n+1}\ \left(0<x<\frac{1}{N-1}\right).\]
By (ii) we have
$$F_1(x_0)=\sum_{n=1}^{\infty}A_n^1x_0^n
=\dfrac{1}{N-1}\sum_{n\in S}((N-1)x_0)^{n+1}=1,$$
which implies that $F(x_0)$ diverges.
Hence the radius of convergence of $F(x)$
does not exceed $x_0$, and we obtain
$h_{\rm top}(\Sigma_{S,N})\ge \log x_0^{-1}
>\log (N-1)$ as required.
\end{proof}

\subsection{Proof of Theorem~C}\label{coded-positive}
   
  Let 
 $S=\{2^n\colon n\in \mathbb{N}\}$ and let
 $N\geq3$ be an integer.
  We consider the fat $S$-gap shift $\Sigma=\Sigma_{S,N}$.
  Lemma~\ref{top-low} gives $h_{\rm top}(\Sigma)>\log (N-1)$.
  Applying the natural decomposition of coded shifts in \eqref{n-dec}, 
we immediately get
$h(\mathcal{C}^p\cup\mathcal{C}^s)=\log (N-1)$, and
hence $h_{\rm top}(\Sigma)>\log (N-1)\geq h_{\rm spec}^{\perp}(\Sigma)$. 
To complete the proof of Theorem~C, it suffices to show that
$h_{\rm spec}^{\perp}(\Sigma)$ is strictly positive.

To this end,
let $\mathcal{L}(\Sigma)=\tilde{\mathcal{C}}^p\tilde{\mathcal{G}}\tilde{\mathcal{C}}^s$ be an arbitrary decomposition
such that ${\tilde{\mathcal{G}}}^M$ has (W)-specification for all $M\in \mathbb{N}$.
The definition of $\Sigma$ gives $w1w\in\mathcal L(\Sigma)$ for all $w\in\{2,\ldots,N\}^*$.
As in the proof of Proposition~\ref{hyp-lem2},
for each $w\in\{2,\ldots,N\}^*\setminus\{\emptyset\}$
we fix once and for all a decomposition 
\[w1w=p(w1w)\cdot c(w1w)\cdot s(w1w),\  \ p(w1w)\in\tilde{\mathcal{C}}^p, c(w1w)\in\tilde{\mathcal{G}}, s(w1w)\in\tilde{\mathcal{C}}^s.\]
Since the symbol $1$ appears in $w1w$ only once,
it appears in
only one of the three words $p(w1w)$, $c(w1w)$, $s(w1w)$.
Note that $p(w1w)$, $c(w1w)$, $s(w1w)$ can be the empty word. 
We will use the following simple estimates for $r\in\{p,s\}$:
\begin{equation}\label{esti-1}0\leq|r(w1w)|\leq |w|\ \text{if $1$ does not appear in $r(w1w)$};\end{equation} 
\begin{equation}\label{esti-2}|w|+1\leq|r(w1w)|\leq 2|w|+1\ \text{if $1$ appears in $r(w1w)$}.\end{equation}

We estimate $h(\tilde{\mathcal{C}}^p\cup\tilde{\mathcal{C}}^s)$ from below by constructing a countably infinite family of finite-to-one maps that take values in $\tilde{\mathcal{C}}^p\cup\tilde{\mathcal{C}}^s$.
We treat two cases separately.
\medskip

\noindent{\it Case 1: $\tilde{\mathcal{G}}\subset
\{2,\ldots,N\}^{\ast}$.}
For each $\ell\in\mathbb N$ 
we define a map
$\Phi_\ell\colon \{2,\ldots,N\}^{2^\ell}\mapsto
\bigcup_{n=2^\ell+1}^{2^{\ell+1}+1} (\tilde{\mathcal{C}}^p_n\cup\tilde{\mathcal{C}}^s_n)$ by
\begin{equation}\label{Phi-1}\Phi_\ell(w)=
\begin{cases}
s(w1w) & \text{if $1$ appears in $s(w1w)$,}\\
p(w1w) & \text{otherwise.}
\end{cases}\end{equation}
 For any $w\in\{2,\ldots,N\}^{2^\ell}$,
the symbol $1$ does not appear in $c(w1w)$
since $\tilde{\mathcal{G}}\subset
\{2,\ldots,N\}^{\ast}$.
By \eqref{esti-2}, 
if $1$ appears in $p(w1w)$ then  $2^\ell+1\leq|p(w1w)|\leq 2^{\ell+1}+1$, and if $1$ appears in $s(w1w)$ then $2^\ell+1\leq|s(w1w)|\leq 2^{\ell+1}+1$. Hence 
$\Phi_\ell$ is well-defined.

For any $w\in\{2,\ldots,N\}^{2^\ell}$, the symbol $1$ appears in $\Phi_\ell(w)$ by the definition \eqref{Phi-1}. 
This implies that $\Phi_\ell$ is injective, and so
\[\#\left(\bigcup_{n=2^\ell+1}^{2^{\ell+1}+1}(\tilde{\mathcal{C}}^p_n\cup\tilde{\mathcal{C}}^s_n)\right)
\ge \#\Phi_\ell(\{2,\ldots,N\}^{2^\ell}) \ge\#\{2,\ldots,N\}^{2^\ell}= (N-1)^{2^\ell}.\] 
Pick $n_\ell\in\{2^\ell+1,\ldots, 2^{\ell+1}+1\}$ such that
\[\#(\tilde{\mathcal{C}}^p_{n_\ell}\cup\tilde{\mathcal{C}}^s_{n_\ell})
\ge \frac{(N-1)^{2^\ell}}{2^\ell+1}.\]
Since $\liminf_{\ell\to\infty}n_\ell=\infty$, we have
\begin{equation}\label{ent-eqI}\begin{split}h(\tilde{\mathcal{C}}^p\cup\tilde{\mathcal{C}}^s)
&\ge \limsup_{l\rightarrow\infty}
\frac{1}{n_\ell}\log
\#(\tilde{\mathcal{C}}^p_{n_\ell}\cup\tilde{\mathcal{C}}^s_{n_\ell})\\
&\ge \limsup_{l\rightarrow\infty}
\frac{1}{n_\ell}\log
\frac{(N-1)^{\frac{n_\ell-1}{2}}}{n_\ell}=\frac{1}{2}\log (N-1).\end{split}\end{equation} 
\medskip

\noindent{\it Case 2: $\tilde{\mathcal{G}}\not\subset
\{2,\ldots,N\}^{\ast}$.} We modify the argument in Case~1. Fix $u\in\tilde{\mathcal{G}}\setminus
\{2,\ldots,N\}^{\ast}$ for which
there exist $k\geq2$, 
$u^{(0)}$, $u^{(k)}\in\{2,\ldots,N\}^{*}$, 
$u^{(1)},\ldots,u^{(k-1)}\in\{2,\ldots,N\}^*\setminus\{\emptyset\}$ such that 
\begin{equation}\label{eq-u}u=u^{(0)}1u^{(1)}1\cdots u^{(k-1)}1u^{(k)}.\end{equation}
Let $t\geq0$ be a gap size for the (W)-specification of $\tilde{\mathcal{G}}$. 
For each $\ell\in\mathbb N$ let \begin{equation}\label{al}
a_\ell=2^\ell-2t-|u^{(k)}|-1.\end{equation}
In what follows we assume
$\ell$ is large enough so that 
\begin{equation}\label{al2}\frac{2a_\ell}{3}>2^{\ell-1}.\end{equation}
We define a map 
$\Psi_\ell\colon \{2,\ldots,N\}^{a_\ell}\mapsto
\bigcup_{n=a_\ell-2^{\ell-1}}^{2a_\ell+1}(\tilde{\mathcal{C}}^p_n\cup\tilde{\mathcal{C}}^s_n)$ by
\begin{equation}\label{Phi-eq}\Psi_\ell(w)=
\begin{cases}
s(w1w) & \text{if $1$ appears in $s(w1w)$,}\\
p(w1w) & \text{otherwise.}
\end{cases}\end{equation}
Compare \eqref{Phi-eq} with \eqref{Phi-1}. To see that $\Psi_\ell$ is well-defined, let $w\in\{2,\ldots,N\}^{a_\ell}$.
By \eqref{esti-2}, if the symbol $1$ appears in $p(w1w)$ then 
$a_\ell+1\leq|p(w1w)|\leq 2a_\ell+1$, and
if 
$1$ appears in $s(w1w)$ then 
$a_\ell+1\leq|s(w1w)|\leq 2a_\ell+1$.
If $1$ appears in $c(w1w)$, then the definition \eqref{Phi-eq} gives $\Psi_\ell(w)=p(w1w)$, and
 the lemma below yields $a_\ell-2^{\ell-1}\leq|p(w1w)|\leq a_\ell$. Hence $\Psi_\ell$ is well-defined.


\begin{lemma}\label{lem-al}If $w\in\{2,\ldots,N\}^{a_\ell}$ and the symbol $1$ appears in $c(w1w)$, then $a_\ell-2^{\ell-1}\leq|p(w1w)|\leq a_\ell$.
\end{lemma}
\begin{proof}
The upper bound follows from \eqref{esti-1}. To prove the lower bound, note that
there exist $v,v'\in\{2,\ldots,N\}^{\ast}$ such that
\[c(w1w)=v1v'.\]
Then we have $w1w=p(w1w)\cdot v1v'\cdot s(w1w)$,
and so $v$ is a subword of $w$ and
 $|p(w1w)|=|w|-|v|$. It suffices to show that $|v|\leq 2^{\ell-1}$.

By \eqref{al} we have 
\begin{equation}
\label{a}
|v|\le |w|= a_\ell<2^\ell-2t-|u^{(k)}|.
\end{equation}
Since $u,c(w1w)\in\tilde{\mathcal{G}}$,
there exists $u'\in\mathcal L(\Sigma)$  such that $|u'|\le t$ and
$uu'c(w1w)\in\mathcal{L}(\Sigma)$, and hence
$1u^{(k)}u'v1\in\mathcal{L}(\Sigma)$.
If $1$ does not appear in $u'$, 
then 
it does not appear in 
$u^{(k)}u'v$ either.
Since $\Sigma$ is the coded shift generated by $\mathcal H$ in \eqref{generator}, it follows that $|u^{(k)}u'v|\in S$.
By \eqref{a}
we have
\[|v|\leq|u^{(k)}u'v|\leq|u^{(k)}|+|u'|+|w|< 2^\ell-t\leq2^\ell.\] 
Since $S$ does not intersect $[2^{\ell-1}+1,2^\ell-1]\cap\mathbb N$, 
we obtain $|v|\leq 2^{\ell-1}$.
If $1$ appears in $u'$, then
let $t'$ denote the word length of the longest suffix of $u'$ that does not contain $1$.
Then we have $t'\leq |u'|\leq t$ and
$1i_{t'}^s(u')v1\in\mathcal{L}(\Sigma)$ and $1$ does not appear in $i_{t'}^s(u')v$.
Since $\Sigma$ is the coded shift generated by $\mathcal H$ in \eqref{generator}, this
implies $|i_{t'}^s(u')v|=t'+|v|\in S$.
Again by \eqref{a} we obtain \[|v|\leq t'+|v|\leq t+|w|<2^\ell.\]
Since $S$ does not intersect $[2^{\ell-1}+1,2^\ell-1]\cap\mathbb N$, 
we obtain $|v|\leq 2^{\ell-1}$ as required.\end{proof}

\begin{lemma}\label{not-inj}
$\Psi_\ell$ is at most $(N-1)^{2^{\ell-1}}$ to $1$.
\end{lemma}

\begin{proof}Let $w,w'\in\{2,\ldots,N\}^{a_\ell}$, $w\neq w'$ and suppose $\Psi_\ell(w)=\Psi_\ell(w')$.
 The definition \eqref{Phi-eq} implies that the symbol $1$ appears in neither $s(w1w)$ nor $s(w'1w')$, and $\Psi_\ell(w)=p(w1w)$, $\Psi_\ell(w')=p(w'1w')$.

If $1$ does not appear in
$c(w1w)$ 
and appears in
$c(w'1w')$, then by \eqref{esti-1}, \eqref{esti-2} we get $|p(w1w)|\geq a_\ell+1$ and 
$|p(w'1w')|\leq a_\ell$, and so $\Psi_\ell(w)\neq \Psi_\ell(w')$, a contradiction.
If $1$ appears in $c(w1w)$ 
and does not appear in
$c(w'1w')$, then
we also obtain a contradiction.
If $1$ appears in neither
$c(w1w)$ 
nor
$c(w'1w')$,
then it appears in both $p(w1w)$ and $p(w'1w')$, and so $w=w'$, a contradiction.
Therefore, $1$ appears in both
$c(w1w)$ 
and 
$c(w'1w')$. By Lemma~\ref{lem-al}
there exists $k\in\{a_\ell-2^{\ell-1},\ldots, a_\ell\}$ such that 
$i^p_{k }(w)=i^p_{k }(w')$.
Therefore, $\Psi_\ell$ is at most $(N-1)^{2^{\ell-1}}$ to $1$.\end{proof}
By Lemma~\ref{not-inj} we have
\[\begin{split}\#\left(\bigcup_{n=a_\ell-2^{\ell-1}}^{2a_\ell+1}(\tilde{\mathcal{C}}^p_n\cup\tilde{\mathcal{C}}^s_n)\right)&\ge\#\Psi_\ell(\{2,\ldots,N\}^{a_\ell})\\
&\geq \#\{2,\ldots,N\}^{a_\ell}(N-1)^{-2^{\ell-1}}=   (N-1)^{a_\ell-2^{\ell-1}}.\end{split}\]
Hence there exists $n_\ell\in\{a_\ell-2^{\ell-1},\ldots, 2a_{\ell}+1\}$
such that \[\#(\tilde{\mathcal{C}}^p_{n_\ell}\cup
\tilde{\mathcal{C}}^s_{n_\ell})\geq
\frac{(N-1)^{a_\ell-2^{\ell-1}} }{
a_\ell+2+2^{\ell-1} }\ge \frac{(N-1)^{
\frac{a_\ell}{3}  } }{
5a_\ell/3+2 }\geq\frac{(N-1)^{\frac{n_\ell-1}{6}}}{5n_\ell+2}.\]
For the last inequality we have used \eqref{al2}. 
Since $\liminf_{\ell\to\infty}n_\ell=\infty$, we have
\begin{equation}\label{ent-eqII}\begin{split}h(\tilde{\mathcal{C}}^p\cup\tilde{\mathcal{C}}^s)
&\ge \limsup_{l\rightarrow\infty}
\frac{1}{n_\ell}\log
\#(\tilde{\mathcal{C}}^p_{n_\ell}\cup\tilde{\mathcal{C}}^s_{n_\ell})\\
&\ge\limsup_{\ell\to\infty}\frac{1}{n_\ell}\log\frac{(N-1)^{\frac{n_\ell-1}{6}}}{5n_\ell+2}\geq\frac{1}{6}\log (N-1).\end{split}\end{equation} Since $\mathcal{L}(\Sigma)=\tilde{\mathcal{C}}^p\tilde{\mathcal{G}}\tilde{\mathcal{C}}^s$ is an arbitrary decomposition
such that ${\tilde{\mathcal{G}}}^M$ has (W)-specification for all $M\in \mathbb{N}$,
from \eqref{ent-eqI} and \eqref{ent-eqII} we obtain
$h_{\rm spec}^{\perp}(\Sigma)\geq(1/6)\log (N-1).$
This completes 
the proof of Theorem~C.\qed
\medskip


\subsection*{Acknowledgments}

MS was supported by the JSPS KAKENHI 21K13816. HT was supported by the JSPS KAKENHI 23K20220.
KY was supported by the JSPS KAKENHI
21K03321.


      \bibliographystyle{amsplain}

\end{document}